\documentclass[12pt,a4paper,intlimits,sumlimits]{amsart}

\usepackage{amsmath, amsthm, mathtools}
\usepackage{amsfonts}
\usepackage[alphabetic,nobysame]{amsrefs}
\usepackage{graphicx}
\usepackage{framed}
\usepackage{amssymb}
\usepackage{esvect}
\usepackage{xcolor}
\usepackage{eucal}
\usepackage{mathrsfs}
\usepackage{hyperref}
\usepackage[english]{babel}
\usepackage{mathrsfs}
\usepackage{esint}

\def \N {\mathbb{N}}
\def \R {\mathbb{R}}

\theoremstyle{definition}
\newtheorem{definition}{Definition}[section]

\newtheorem{remark}[definition]{Remark}

\theoremstyle{plain}
\newtheorem{theorem}[definition]{Theorem}
\newtheorem{proposition}[definition]{Proposition}
\newtheorem{lemma}[definition]{Lemma}
\newtheorem{corollary}[definition]{Corollary}

\newtheorem{openproblem}{Open Problem}

\numberwithin{equation}{section}

\usepackage{geometry}
\renewcommand{\epsilon}{\varepsilon}

\renewcommand{\leq}{\leqslant}
\renewcommand{\le}{\leqslant}
\renewcommand{\geq}{\geqslant}
\renewcommand{\ge}{\geqslant}

\allowdisplaybreaks

 \title[$(s, p)$-superposition of nonlinear fractional operators]{A general theory \\ for the $(s, p)$-superposition \\ of nonlinear fractional operators}

\author[S. Dipierro, E. Proietti Lippi, C. Sportelli and E. Valdinoci]{Serena Dipierro, Edoardo Proietti Lippi, Caterina Sportelli and Enrico Valdinoci}

\address{Department of Mathematics and Statistics
\newline\indent University of Western Australia \newline\indent
35 Stirling Highway, WA 6009 Crawley, Australia.\newline
\newline\indent
\tt serena.dipierro@uwa.edu.au \newline\indent
\tt edoardo.proiettilippi@uwa.edu.au \newline\indent
\tt caterina.sportelli@uwa.edu.au \newline\indent
\tt enrico.valdinoci@uwa.edu.au}

\begin{document}

\maketitle

\begin{abstract}
We consider the continuous superposition of operators of the form
\[
\iint_{[0, 1]\times (1, N)} (-\Delta)_p^s \,u\,d\mu(s,p),
\]
where~$\mu$ denotes a signed measure over the set~$[0, 1]\times (1, N)$, joined to a nonlinearity satisfying a proper subcritical growth.  The novelty of the paper relies in the fact that, differently from the existing literature, the superposition occurs in both~$s$ and~$p$.

Here we introduce a new framework which is so broad to include, for example,  the scenarios of the finite sum of different (in both~$s$ and~$p$) Laplacians, or of a fractional $p$-Laplacian plus a $p$-Laplacian, or even combinations involving some fractional Laplacians with the ``wrong" sign.

The development of this new setting comes with two applications, which are related to the Weierstrass Theorem and a Mountain Pass technique. The results obtained contribute to the existing literature with several specific cases of interest which are entirely new.
\end{abstract}

\tableofcontents

\section{Introduction}

\subsection{Framework of the problem}
In this paper, we introduce a new functional framework tailored to address the
case of equations involving the superposition of fractional $p$-Laplace operators of different orders under Dirichlet boundary conditions.  

The issue of the superposition of (possibly fractional) operators, namely operators of the form
\[
\int_{[0, 1]} (-\Delta)^su \, d\mu(s),
\]
for a suitable measure~$\mu$ on the interval of fractional exponents~$[0, 1]$, is a topic under extensive investigation: see, for instance, \cite{MR3485125}, in which
the measure on fractional exponents was supposed to be positive and supported away from~$s = 0$. 
A general setting (also removing these two assumptions)
was recently presented in~\cite{MR4736013}, where, in addition, the case of signed measures\footnote{Accounting for negative signs in diffusive operators is not just a mathematical curiosity, since
elliptic operators with negative sign can be useful to model concentration phenomena (e.g., in the heat equation with inverted time direction). For example, in mathematical biology, a simple model in which a Laplacian ``with the wrong sign'' naturally appears is that of a biological species subject to chemotaxis when the chemotactic agent is proportional to the logarithm of the density of population (see equation~(1.1.21) in~\cite{ZZLIB} with~$w:=\ln u$).} was taken into account, that is when~$\mu=\mu^+ -\mu^-$, with~$\mu^+$ and~$\mu^-$ (nonnegative) measures. This setting has been also extended to the case of superposition of fractional $p$-Laplacians, for a given Lebesgue exponent~$p$ (see~\cite{DPSV2}), and, more recently, to the superposition of fractional operators under Neumann boundary conditions (see~\cite{TUTTI, TUTTI2}).

The present work enriches the existing literature with a case never addressed before, in which the superposition occurs on both the exponents~$s$ and~$p$.

The results obtained are very comprehensive but they are also new in a number of specific interesting cases which will be acquired as an easy consequence of our overarching strategy.

Let now focus on the following specific framework.
Assume that~$N\ge 2$. We consider two finite (Borel) measures~$\mu^+$ and~$\mu^-$ over the set
\begin{equation}\label{Sigma}
\Sigma:=[0,1]\times (1,N)
\end{equation}
and we assume that, for some~$\overline{s}\in (0,1]$, they satisfy
\begin{equation}\label{ipotesimu11}
\mu^+ \Big([\overline s,1]\times (1, N)\Big)>0\quad\mbox{ and }\quad \mu^- \Big([\overline s,1]\times (1, N)\Big)=0.
\end{equation}
\vspace{-.35in}
\begin{figure}[h]
\begin{center}
\includegraphics[scale=.45]{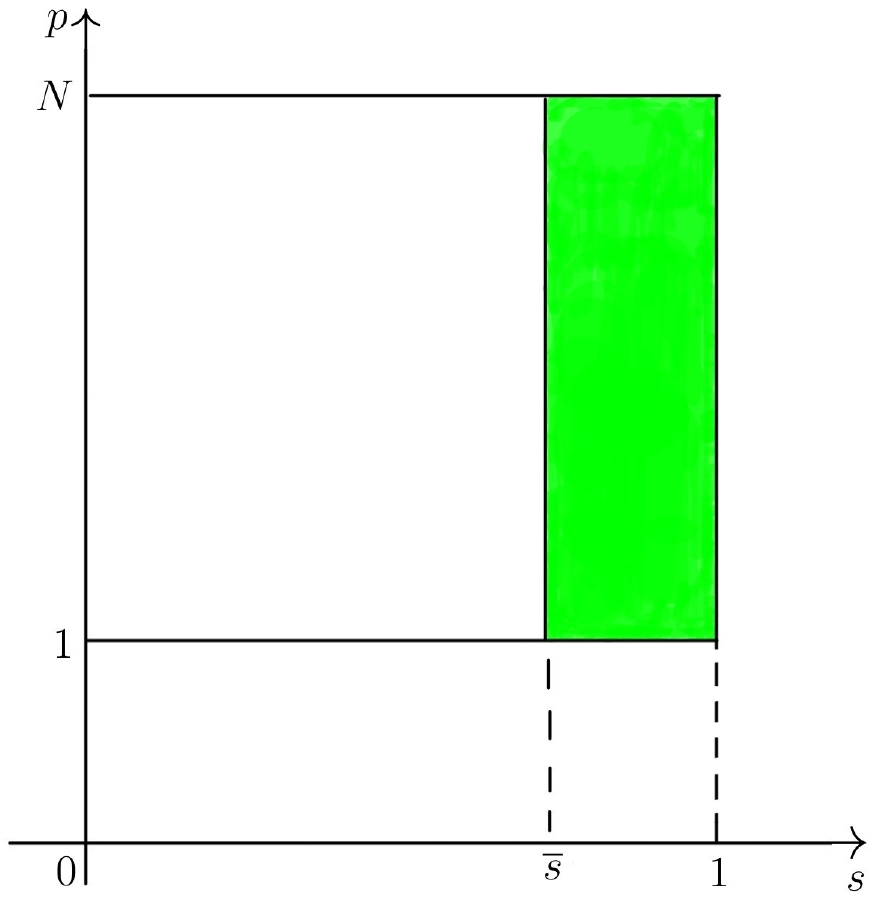}
\end{center}
\caption{The set~$[\overline s,1]\times (1, N)$ (in green) introduced in~\eqref{ipotesimu11}.}
\label{green}
\end{figure}

Under these assumptions, we consider the signed measure~$\mu$ given by
\begin{equation}\label{mudefinition}
\mu:=\mu^+-\mu^-.
\end{equation}
The main operator of interest for us takes the form
\begin{equation}\label{elle}
L_\mu u:=\iint_\Sigma (-\Delta)_p^s \,u\,d\mu(s,p).
\end{equation}

We point out that the above double integration may happen with variable signs, since~$\mu$ is a signed measure with possibly negative components induced by~$\mu^-$. Nonetheless, condition~\eqref{ipotesimu11} requires that the negative components do not impact the set~$[\overline s,1]\times (1, N)$ (corresponding, roughly speaking, to the ``large'' fractional exponents). This feature is illustrated in Figure~\ref{green}.

In our setting, when~$s\in(0,1)$, the fractional~$p$-Laplacian~$(- \Delta)_p^s$ is the nonlinear nonlocal operator defined on smooth functions by
\[
(- \Delta)_p^s\, u(x) := c_{N,s,p} \,\lim_{\varepsilon \searrow 0} \int_{\R^N \setminus B_\varepsilon(x)} \frac{|u(x) - u(y)|^{p-2}\, (u(x) - u(y))}{|x - y|^{N+sp}}\, dy.
\]
The exact value of the positive normalizing constant~$c_{N,s,p}$ is not important for us, except for allowing
a consistent setting for the limit cases, namely that
\[
\lim_{s\searrow0}(- \Delta)_p^s\, u=(- \Delta)_p^0\, u:=u
\]
and
\[
\lim_{s\nearrow1}(- \Delta)_p^s\, u=(- \Delta)_p^1\, u:=-\Delta_pu=-{\operatorname{div}}(|\nabla u|^{p-2}\nabla u).
\]
However, one can take
\begin{equation}\label{cnsp}
c_{N,s,p}:=\frac{s\,2^{2s}\,\Gamma\left(\frac{ps+p+N-2}{2}\right)}{\pi^{N/2}\,\Gamma(1-s)},
\end{equation}
see e.g.~\cite[page~130]{MR3473114} and the references therein.

In this setting, we consider a bounded open set~$\Omega\subset\R^N$
with smooth boundary.
Our aim is to give existence results for problems of the type
\begin{equation}\label{P1} 
\begin{cases}
L_\mu u= f(x, u) &{\mbox{ in }}\Omega,\\
u=0&{\mbox{ in }}\R^N\setminus\Omega,
\end{cases}
\end{equation}
being~$L_\mu$ as in~\eqref{elle}. Here, $f:\Omega\times \R \to\R$ is a function satisfying some suitable assumptions that we will list later.

In addition to~\eqref{mudefinition}, we assume that there exists a signed measure~$\mu_s:=\mu_s^+-\mu_s^-$ and a nonnegative measure~$\mu_p$, defined
respectively on~$[0,1]$ and~$(1,N)$, such that
\begin{equation}\label{ipotesimuprodotto}
\mu=\mu^+-\mu^-
=\mu_s^+ \times \mu_p - \mu_s^- \times \mu_p.
\end{equation}

A special case of problem~\eqref{P1} occurs when~$\mu^+$ has a ``discrete'' structure, consisting of the sum of finitely many Dirac's measures.
In light of~\eqref{ipotesimuprodotto}, both~$\mu_s^+$ and~$\mu_p$
have to be a finite sum of Dirac's measures. 

More precisely, let~$n$, $m\in\N\setminus\{0\}$, $s_1,\dots, s_n\in [0, 1]$ and~$p_1,\dots, p_m\in (1, N)$ and take
\begin{equation}\label{muvbdshail}
\mu:=\sum_{{i=1,\dots,n}\atop{j=1,\dots,m }}\delta_{(s_i, p_j)} 
-\mu^-_s\times\sum_{j=1}^m \delta_{p_j}.
\end{equation}

In this case, we will first look at the problem
\begin{equation}\label{problemaW} 
\begin{cases}
\displaystyle \sum_{{i=1,\dots,n}\atop{j=1,\dots,m }} (-\Delta)^{s_i}_{p_j} u -\displaystyle\sum_{j=1}^m\,\int_{[0,1]} (-\Delta)^s_{p_j} u \, d\mu_s^-(s) = g(x) &{\mbox{ in }}\Omega,\\
u=0&{\mbox{ in }}\R^N\setminus\Omega.
\end{cases}
\end{equation}

Moreover, we set
\begin{equation}\label{defphatr843967}
\widehat{p}:=\max\{p_1,\dots, p_m\}
\end{equation}
and we assume that~$g\in L^{\widehat{p}_\star}(\Omega)$, being~$\widehat{p}_\star$ the H\"older conjugate of~$\widehat{p}$.

Reducing ourselves to the study of~\eqref{problemaW} rather than the more general~\eqref{P1} is often
necessary in our setting since the case in which~$\mu^+$ is, say, a convergent series of Dirac's measures
offers additional difficulties at the level of functional analysis. Specifically,
in such a level of generality, the associated energy functional may not be
well defined on the naturally associated functional space, as we will clarify with explicit examples (see Appendix~\ref{appendice}).

\subsection{Main results}

Our first set of results deal with problem~\eqref{problemaW}.
The main construction in this setting is given by the following. We consider the measure~$\mu$ defined in~\eqref{muvbdshail}.

We point out that, taking~$\overline{s}:=\max\{s_1,\dots,s_n\}$,
we have that
$$\mu^+ \Big([\overline s,1]\times (1, N)\Big)
=\delta_{\overline s}\times \sum_{j=1}^{m }\delta_{ p_j}>0,$$
and therefore 
the assumption in~\eqref{ipotesimu11}
boils down to asking that
\begin{equation}\label{anchequestat4390}
\mu^-_s([\overline s,1])=0.
\end{equation}

Moreover, we assume that
there exist some constants~$\delta$, $\eta>0$ and~$\gamma>0$ sufficiently small (to be taken as in~\eqref{gammabound}), such that
\begin{equation}\label{ipotesimus}
\mu_s^-([0,\overline{s}) )\le \gamma \mu_s^+([\overline{s},1])
\end{equation}
and
\begin{equation}\label{ipotesimup}
\mbox{supp}(\mu_p)\subset[1+\delta,N-\eta].
\end{equation}

In this setting, the following result holds true:

\begin{theorem}\label{Weierstrass}
There exists~$\gamma_0>0$ such that, if~$\gamma\in[0,\gamma_0]$, the
following statement holds true.

Let~$\mu$ be as in~\eqref{muvbdshail} and $g\in L^{\widehat{p}_\star}(\Omega)$. Assume that~\eqref{anchequestat4390}, \eqref{ipotesimus} and~\eqref{ipotesimup}
hold true.

Then, there exists a weak solution of problem~\eqref{problemaW} corresponding to a global minimizer of the associated energy functional.

Moreover, if~$\mu^-\equiv0$, then the solution is unique.
\end{theorem}

The wide generality of Theorem~\ref{Weierstrass} allows us to provide the existence of a nontrivial solution in many  particular cases of interest, each of which is new in the literature.

\begin{corollary}\label{cor1}
Let~$0<s_2 < s_1 \le1$ and~$1<p<N$. Let~$\alpha\in\R$ and~$g\in L^{\frac{p}{p-1}}(\Omega)$.

Then, there exists~$\alpha_0>0$ depending on~$N$ and~$\Omega$ such that if~$\alpha<\alpha_0$, then the problem
\[
\begin{cases}
(-\Delta)^{s_1}_p u -\alpha(-\Delta)^{s_2}_p u= g(x) &{\mbox{ in }}\Omega,\\
u=0&{\mbox{ in }}\R^N\setminus\Omega
\end{cases}
\]
admits a nontrivial solution corresponding to a global minimizer of the associated energy functional. 

In addition, the problem
\[
\begin{cases}
(-\Delta)^{s_1}_p u = g(x) &{\mbox{ in }}\Omega,\\
u=0&{\mbox{ in }}\R^N\setminus\Omega
\end{cases}
\]
has a unique solution.
\end{corollary}

\begin{corollary}\label{cor3}
Let~$s_k$ be a sequence of real numbers such that~$0\leq s_k<1$ for all~$k\in \N$.
Let~$\alpha_k$ be a sequence of real numbers.
Let~$1<p<N$ and
let~$g\in L^{\frac{p}{p-1}}(\Omega)$.

Then, there exists~$\alpha_0>0$ depending on~$N$ and~$\Omega$ such that if 
\[
\sum_{k=1}^{+\infty} \alpha_k<\alpha_0,
\]
then the problem
\[
\begin{cases}
\displaystyle-\Delta_p u - \sum_{k=1}^{+\infty}\alpha_k(-\Delta)^{s_k}_p u= g(x) &{\mbox{ in }}\Omega,\\
u=0&{\mbox{ in }}\R^N\setminus\Omega
\end{cases}
\]
admits a nontrivial solution corresponding to a global minimizer of the associated energy functional. 
\end{corollary}

\begin{corollary}\label{cor4}
Let~$\omega:[0,1] \to \R$ be any measurable function.
Let~$1<p<N$ and
let~$g\in L^{\frac{p}{p-1}}(\Omega)$.

Then, there exists~$\alpha_0>0$ depending on~$N$ and~$\Omega$ such that if 
\[
\int_0^1 \omega(s)\,ds<\alpha_0,
\]
then the problem
\[
\begin{cases}
\displaystyle-\Delta_p u - \, \int_0^1 \omega(s) (-\Delta)^{s}_{p} u \, ds= g(x) &{\mbox{ in }}\Omega,\\
u=0&{\mbox{ in }}\R^N\setminus\Omega
\end{cases}
\]
admits a nontrivial solution corresponding to a global minimizer of the associated energy functional.
\end{corollary}

The second set of our results deal with the case in which
the negative part of the measure~$\mu$ vanishes,
namely~$\mu^-\equiv 0$. In this circumstance, we address the problem
\begin{equation}\label{senzamumeno} 
\begin{cases}
\displaystyle\sum_{k=1}^m (-\Delta)^{s_k}_{p_k} u = f(x, u) &{\mbox{ in }}\Omega,\\
u=0&{\mbox{ in }}\R^N\setminus\Omega.
\end{cases}
\end{equation}
Here~$m\in\N\setminus\{0\}$, $s_1,\dots, s_m\in [0, 1]$ and~$p_1,\dots, p_m\in (1, N)$. 
Moreover, we assume that~$f:\Omega\times\R\to\R$ is a Carath\'eodory function 
satisfying a suitable subcritical growth.

More precisely, we pick~$s_\sharp\in\{s_1,\dots,s_m\}$
and~$ p_\sharp\in\{p_1,\dots, p_m\}$ 
in order to maximize the critical exponent\footnote{We point out that, since~$s_\sharp\in[0,1]$ and~$p_\sharp\in(1,N)$, we have that~$N-s_\sharp p_\sharp\ge N-p_\sharp>0$.}
\begin{equation*}
(p_\sharp)^{*}_{s_\sharp} := \frac{N p_\sharp}{N-s_\sharp p_\sharp}.
\end{equation*}
In this way, we have that, for all~$k\in\{1,\dots,m\}$,
$$ p_k< \frac{N p_k}{N-s_k p_k}\le \frac{N p_\sharp}{N-s_\sharp p_\sharp}
=(p_\sharp)^{*}_{s_\sharp}$$
and therefore, recalling~\eqref{defphatr843967},
$$ \widehat{p}<(p_\sharp)^{*}_{s_\sharp}.$$

With this notation,
we require~$f$ to satisfy the following conditions:
\begin{equation}\label{AR1}
\begin{split}
&\mbox{there exist some constants~$a_1$, $a_2>0$ and~$q\in (\widehat{p}, (p_\sharp)^*_{s_\sharp})$ such that} \\
&|f(x,t)|\leq a_1+a_2|t|^{q-1} \quad\mbox{for any } t\in \R \mbox{ and for a.e. } x\in \Omega,\\
\end{split}
\end{equation}
and
\begin{equation}\label{AR2}
\lim_{t\to 0}\frac{f(x,t)}{|t|^{\widehat{p}-2} \,t}=0  \quad \mbox{ uniformly for
a.e. } x\in \Omega.
\end{equation}
Moreover, setting
\begin{equation}\label{definizioneF}
F(x,t):=\int_0^t f(x,\tau)\,d\tau,
\end{equation}
we assume that the Ambrosetti--Rabinowitz condition holds true, i.e.
\begin{equation}\label{AR3}
\begin{split}
&\mbox{there exist~$\vartheta>\widehat{p}$ and~$r\geq 0$ such that}\\
&0<\vartheta F(x,t)\leq f(x,t)t  \quad\mbox{for any } t \mbox{ such that } |t|> r \mbox{ and for a.e. } x\in \Omega.
\end{split}
\end{equation}
In addition, we suppose that 
\begin{equation}\label{AR4}
\begin{split}
&\mbox{there exist~$\widetilde{\vartheta}>\widehat{p}$, $a_3>0$ and~$a_4\in L^1(\Omega)$ such that}\\
&F(x,t)\geq a_3|t|^{\widetilde{\vartheta}}-a_4(x).
\end{split}
\end{equation}
The result that we prove in this context reads as follows:

\begin{theorem}\label{MPT}
Let~$f:\Omega\times \R \to \R$ satisfy~\eqref{AR1}, \eqref{AR2}, \eqref{AR3} and~\eqref{AR4}.

Then, there exists a nontrivial weak solution of problem~\eqref{senzamumeno} of Mountain Pass type.
\end{theorem}

We stress that, the situation considered in this paper is new, since, to the best of our knowledge, the literature on problem~\eqref{senzamumeno} is limited to the particular case in which~$m=2$, $p_2>p_1$ and~$s_2>s_1$, see e.g.~\cite{MR3910033}, and
our approach allows us to deal with cases which had remained uncovered. For instance,  we present the following scenario which comes as an easy consequence of Theorem~\ref{MPT}.

\begin{corollary}\label{cor5}
Let~$f:\Omega\times \R \to \R$ satisfy~\eqref{AR1}, \eqref{AR2}, \eqref{AR3} and~\eqref{AR4}. Let~$0\le s_1<s_2\le 1$ and~$1<p_2<p_1<N$.

Then, the problem
\begin{equation*}
\begin{cases}
(-\Delta)^{s_1}_{p_1}\, u +  (-\Delta)^{s_2}_{p_2} \, u = f(x, u) &{\mbox{ in }}\Omega,\\
u=0&{\mbox{ in }}\R^N\setminus\Omega
\end{cases}
\end{equation*}
admits a nontrivial weak solution of Mountain Pass type.
\end{corollary}

The broad generality of our setting allows to consider problem~\eqref{senzamumeno} when the points~$(s_k,p_k)$ all share the same critical exponent, namely the condition
\begin{equation}\label{criticocostante}
\frac{N p_k}{N-s_k p_k} \ \mbox{ is constant for any } k\in\{1,\dots, m\}
\end{equation}
holds. This interesting situation is illustrated in the forthcoming Figure~\ref{punti}.

\begin{corollary}\label{cor6}
Let~$f:\Omega\times \R \to \R$ satisfy~\eqref{AR1}, \eqref{AR2},  \eqref{AR3} and~\eqref{AR4}.  Let~$(s_k, p_k)$ satisfy condition~\eqref{criticocostante}.

Then,  problem~\eqref{senzamumeno} admits a weak solution of Mountain Pass type.
\end{corollary}

\subsection{Further comments and remarks}
The study of problems~\eqref{P1} and~\eqref{problemaW} goes through some preliminary steps.  It is worthwhile to mention that, in this general setting, the choice of a proper functional space for the variational argument to take place is nontrivial, due to the presence of the measure~$\mu$.
Indeed,  we show that some preliminary properties of the functional space (namely the completeness and some embedding results) can be stated even under the ``lighter" assumption that
\begin{equation}\label{ipotesimu1}
\mbox{there exists~$(\overline{s}, \overline{p})\in (0,1]\times (1,N)$ such that } \mu^+\Big([\overline{s},1]\times[\overline{p},N)\Big)>0.
\end{equation}

\begin{figure}[h]
\begin{center}
\includegraphics[scale=.45]{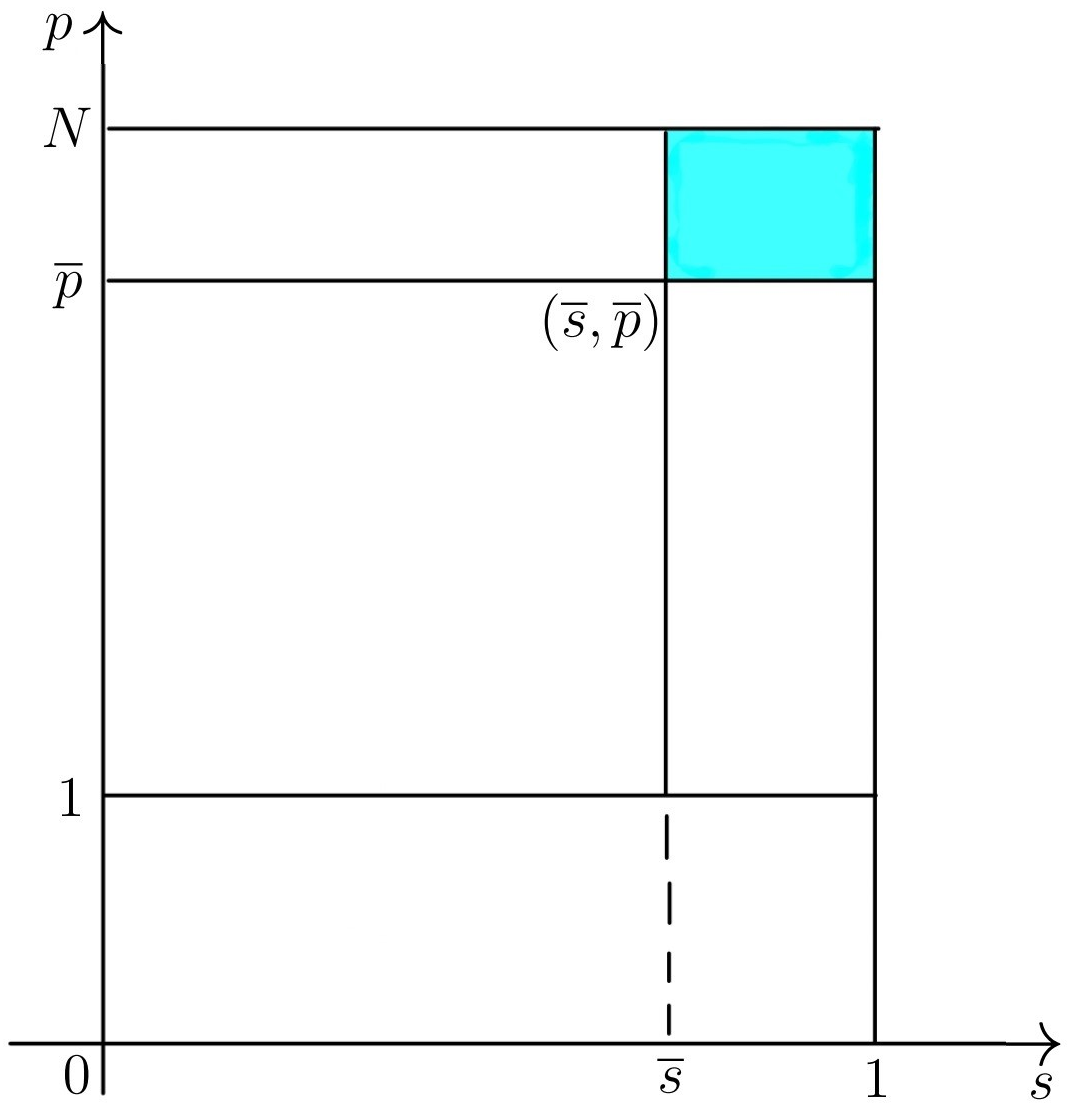}
\end{center}
\caption{The set~$[\overline{s},1]\times[\overline{p},N)$ (in blue) introduced in~\eqref{ipotesimu1}.}
\label{blu}
\end{figure}

We point out that hypothesis~\eqref{ipotesimu1} is weaker than the condition stated in~\eqref{ipotesimu11}.
Indeed, if~\eqref{ipotesimu11} is in force, then there exists~$\overline{p}$ such that~$\mu^+([\overline{s},1]\times[\overline{p},N))>0$,
hence~\eqref{ipotesimu1} is satisfied as well.  One can also compare assumptions~\eqref{ipotesimu11} and~\eqref{ipotesimu1}
through Figures~\ref{green} and~\ref{blu}.

Also,
at this preliminary stage, no assumption on the negative part of the measure~$\mu$ is required, as the functional space is defined by means of~$\mu^+$, in a way which will be fully clarified in Section~\ref{sec2}.

Due to the wide generality of our setting, things get more complicated when addressing the issue of the functional space being uniformly convex or not. To not weigh down too much this introduction, we refer the interested reader to the discussion handled in Open Problem~\ref{OP1}
on page~\pageref{OP1}.
However, the difficulty related to the possible lack of uniform convexity (or reflexivity)
compel us to restrict ourselves, on some occasions, to the case of the positive part of the measure~$\mu^+$ being a convergent series of Dirac's measures.
We point out that our results remain new in the literature even under this more restrictive condition.
We believe that the case of a general measure only satisfying~\eqref{ipotesimu11} is an interesting open problem to be investigated in the future. 

To further develop our analysis, we need the positive and negative parts of the measure~$\mu$ to interact properly, namely the contribution coming from~$\mu^-$ has to be “reabsorbed” through quantitative estimates into~$\mu^+$. To this extent, we will use
the hypotheses~\eqref{ipotesimuprodotto}, \eqref{ipotesimus}
and~\eqref{ipotesimup}.

\subsection{Organization of the paper}
Our paper is organized as follows. The functional framework needed for this paper will be presented in Section~\ref{sec2}.
Section~\ref{sec3} will deal with the proof of the main Theorems~\ref{Weierstrass} and~\ref{MPT}.
Section~\ref{sec4} will contain some useful applications of our main results, while Appendix~\ref{appendice} contains an example 
of measure~$\mu$ for which the variational formulation of
problem~\eqref{P1} may be rather problematic. 

\section{Functional setting and preliminary results}\label{sec2}

In this section we introduce the variational framework needed to address problem~\eqref{P1}. 

\subsection{The functional setting}
We introduce the following notation:
\[
[u]_{s,p}:=
\begin{cases}
\|u\|_{L^p(\R^N)}  &\mbox{ if } s=0,
\\ \\
\displaystyle\left(c_{N,s,p}\iint_{\R^{2N}}\frac{|u(x)-u(y)|^p}{|x-y|^{N+sp}}\,dx\,dy \right)^{1/p} &\mbox{ if } s\in(0,1),
\\ \\
\|\nabla u\|_{L^p(\R^N)}  &\mbox{ if } s=1.
\end{cases}
\]
In light of this and the definition in~\eqref{cnsp}, we have that
\[
\lim_{s\searrow0}[u]_{s,p}=[u]_{0,p}\qquad{\mbox{and}}\qquad\lim_{s\nearrow1}[u]_{s,p}=[u]_{1,p}.
\]
Let~$\mu^+$ be a finite (Borel) measure satisfying~\eqref{ipotesimu1}. We set
\begin{equation}\label{definizionenorma}
\|u\|_\mu:=\iint_\Sigma [u]_{s,p}\,d\mu^+(s,p).
\end{equation}

The definition in~\eqref{definizionenorma} might appear ``unusual". Indeed, one could be tempted to consider instead the quantity
\[
\iint_\Sigma [u]^p_{s,p}\,d\mu^+(s,p),
\]
but the latter does not define a norm, due to lack of homogeneity.

We show instead that~\eqref{definizionenorma} does define a norm:

\begin{proposition}
The function~$\|\cdot\|_\mu$ in~\eqref{definizionenorma} defines a norm.
\end{proposition}

\begin{proof}
The triangle inequality and the absolute homogeneity follow from the fact
that these properties hold true for~$[\cdot]_{s,p}$.

We now check the positivity property. Clearly, we have that~$\|u\|_\mu\ge0$.

Moreover, let~$u\in\mathcal{X}_\mu (\Omega)$ such that~$\|u\|_\mu=0$, namely
\[
\iint_\Sigma [u]_{s,p}\,d\mu^+(s,p)=0.
\]
We show that this entails that~$u\equiv 0$ in~$\R^N$.
Notice that, in virtue of~\eqref{ipotesimu1}, at least one of the following cases occurs:
\begin{itemize}
\item[i)] if~$(0, p)\in\mbox{supp}(\mu^+)$ for some~$p\in (1, N)$, then
\[
\|u\|_{L^p(\Omega)} =0;
\]
\item[ii)] if~$(s, p)\in \mbox{supp}(\mu^+)$
for some~$s\in(0,1)$ and~$p\in (1, N)$, then
\[
\iint_{\R^{2N}}\frac{|u(x)-u(y)|^p}{|x-y|^{N+sp}}\,dx\,dy=0;
\]
\item[iii)] if~$(1, p)\in\mbox{supp}(\mu^+)$ for some~$p\in (1, N)$, then
\[
\|\nabla u\|_{L^p(\Omega)} =0.
\]
\end{itemize}
In case~$i)$, we have that~$u\equiv0$ in~$\Omega$, and therefore~$u\equiv0$ in~$\R^N$.

In case~$ii)$, we have that~$u$ is constant a.e. in~$\R^N$. Since~$u\equiv0$ in~$\R^N\setminus\Omega$,
we have that~$u\equiv0$ a.e. in~$\R^N$. 

In case~$iii)$, we use the Poincar\'e inequality to see that
$$ \|u\|_{L^p(\Omega)}\le C
\|\nabla u\|_{L^p(\Omega)}=0,$$
for some~$C>0$. This implies that~$u\equiv0$ a.e. in~$\Omega$, and therefore in~$\R^N$, as desired.
\end{proof}

We now define~$\mathcal{X}_\mu (\Omega)$ as the completion
of~$C^\infty_0(\Omega)$ with respect to the norm in~\eqref{definizionenorma}.

We stress that we have only assumed~\eqref{ipotesimu1} so far.
This condition alone will allow us to establish some embedding results for
the space~$\mathcal{X}_\mu (\Omega)$
(see the forthcoming Proposition~\ref{propembedding}).

\subsection{Some embedding results}
It will be useful in the forthcoming results to observe that, by~\eqref{ipotesimu1}, 
\begin{equation}\label{ipotesisharp}
\mbox{there exists~$(s_\sharp, p_\sharp)\in[\overline{s},1]\times [\overline{p},N)$ such that~$\mu^+([s_\sharp,1]\times[p_\sharp,N))>0$.}
\end{equation}

We emphasize that one can have more than one couple~$(s_\sharp, p_\sharp)$ which satisfies~\eqref{ipotesisharp}. If this is the case,  then one can pick~$s_\sharp$ and~$p_\sharp$ in order to have 
\begin{equation}\label{crit}
(p_\sharp)^{*}_{s_\sharp} := \frac{N p_\sharp}{N-s_\sharp p_\sharp}
\end{equation}
as large as possible, but still fulfilling condition~\eqref{ipotesisharp}.

It is worth noting that, in our setting, one may have different couples~$(s_\sharp, p_\sharp)$ which realize~\eqref{crit}. More precisely, one can consider the function
\[
\Sigma \ni(s,p)\mapsto \frac{Np}{N-sp}
\]
and observe that, by construction, the level sets of this function are made of points with the same critical exponent. It is not hard to show that such level sets are the intersection between~$\Sigma$ and hyperbolas of the form
\[
p(s)=\frac{CN}{N+Cs}.
\]
Here, the constant~$C$ is the critical exponent of the points belonging to the level set, namely
\[
C=\frac{Np(s)}{N-sp(s)}.
\]
This is illustrated in Figure~\ref{punti}.

\begin{figure}[h]
\begin{center}
\includegraphics[scale=.40]{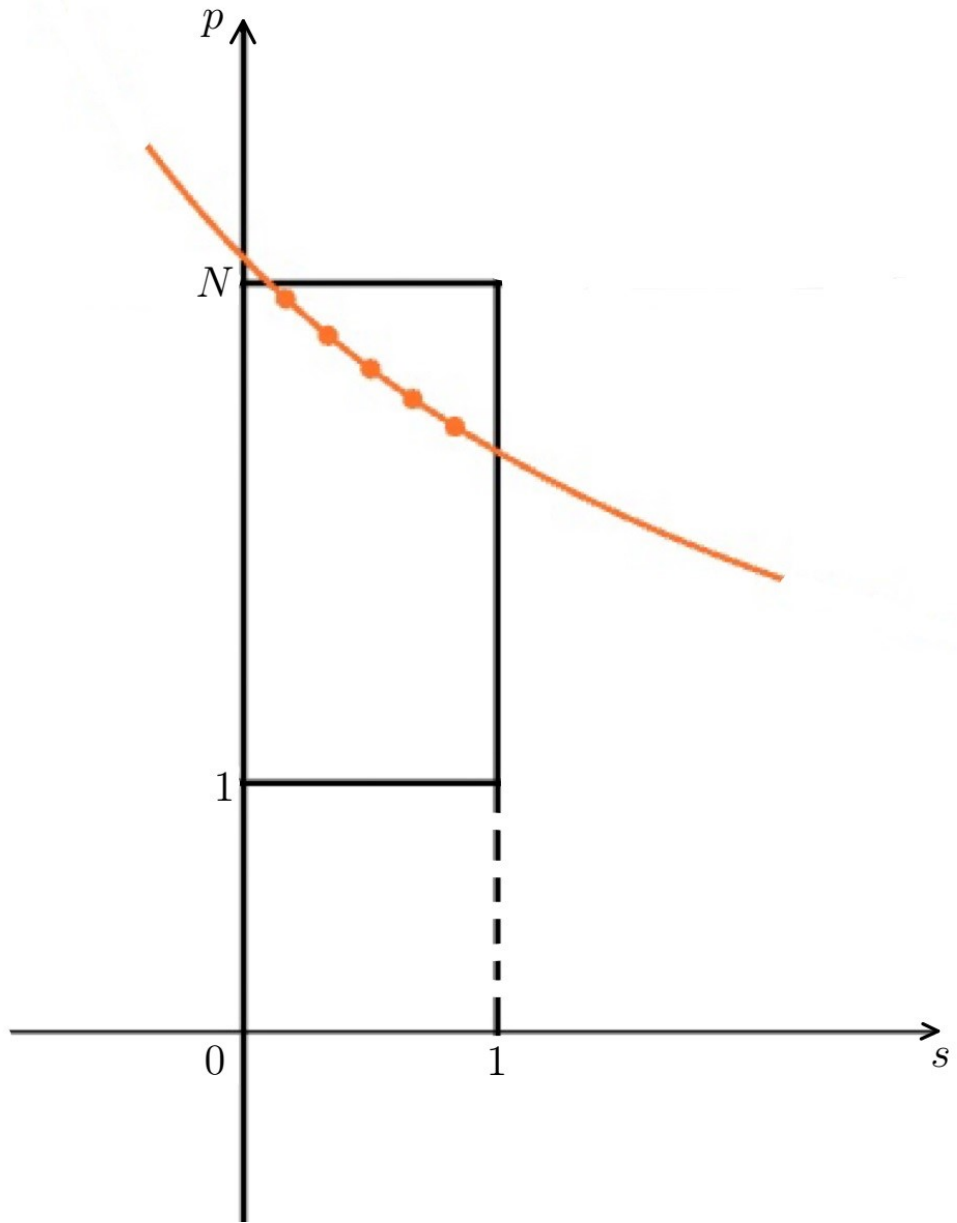}
\end{center}
\caption{The points~$(s_\sharp, p_\sharp)$ sharing the same critical exponent~$(p_\sharp)^{*}_{s_\sharp}$.}
\label{punti}
\end{figure}

We now aim to prove that~$\mathcal{X}_\mu (\Omega)$ is embedded in suitable Sobolev spaces. To this end, we first recall the following result established in~\cite[Theorem~3.2]{DPSV2}.

\begin{theorem}\label{sobolevmio}
Let~$\Omega$ be a bounded, open subset of~$\R^N$ and let~$p\in (1,N)$.

Then, there exists~$C=C(N,\Omega,p)>0$ such that, for every~$s$, $S\in[0,1]$ with~$s\le S$ and every measurable function~$u:\R^N\to\R$ with~$u=0$ a.e. in~$\R^N\setminus\Omega$,
\[
[u]_{s,p}\le C\,[u]_{S,p}.
\]
\end{theorem}

We are now in the position to provide the desired embedding result.

\begin{proposition}\label{propembedding}
Let~$s_\sharp$ and~$p_\sharp$ be given by~\eqref{ipotesisharp}. Then,
\begin{itemize}
\item if~$s_\sharp=1$, then there exists~$C=C(p_\sharp, N, \Omega,\mu^+)>0$
such that, for any~$u\in \mathcal{X}_\mu (\Omega)$,
\[
\| u\|_{W_0^{1,p_\sharp}(\Omega)}\le C\, \|u\|_\mu,
\]
namely the space~$\mathcal{X}_\mu (\Omega)$ is continuously embedded in~$W_0^{1,p_\sharp}(\Omega)$;
\item if~$s_\sharp\in [\overline s,1)$, then there exists~$C=C(s_\sharp, p_\sharp, N, \Omega,\mu^+)>0$ such that, for any~$u\in \mathcal{X}_\mu (\Omega)$,
\[
\|u\|_{L^{(p_\sharp)^{*}_{s_\sharp}}(\Omega)}\le C\, \|u\|_\mu,
\]
namely the space~$\mathcal{X}_\mu (\Omega)$ is continuously embedded in~$L^{(p_\sharp)^{*}_{s_\sharp}}(\Omega)$.

In addition, if there exists~$(s_1, p_1)\in (0,1)\times (1,N)$ such that~$\mu^+([s_1,1)\times\{p_1\})>0$, then there exists~$C=C(s_1, p_1, N, \Omega,\mu^+)>0$ such that, for any~$u\in \mathcal{X}_\mu (\Omega)$,
\[
\|u\|_{W^{s_1,p_1}_0(\Omega)}\le C\, \|u\|_\mu,
\]
namely the space~$\mathcal{X}_\mu (\Omega)$ is continuously embedded in~$W_0^{s_1,p_1}(\Omega)$.
\end{itemize}
\end{proposition}

\begin{proof}
We assume first that~$s_\sharp=1$. By~\eqref{ipotesisharp} we see that 
\[
\mbox{there exists~$\eta>0$ such that } \mu^+(\{1\}\times [p_\sharp,N-\eta])>0.
\]
Then, from this and~\eqref{definizionenorma} we deduce that
\[
\begin{aligned}
\|u\|_\mu&=\iint_\Sigma [u]_{s,p}\,d\mu^+(s,p) \\
&\ge \iint_{\{1\}\times [p_\sharp,N-\eta]} \|\nabla u\|_{L^p(\Omega)}\,d\mu^+(s,p) \\
&\ge \iint_{\{1\}\times [p_\sharp,N-\eta]} |\Omega|^{\frac{1}{p}-\frac{1}{p_\sharp}} \|\nabla u\|_{L^{p_\sharp}(\Omega)}\,d\mu^+(s,p) \\
&= \iint_{\{1\}\times [p_\sharp,N-\eta]} |\Omega|^{\frac{1}{p}-\frac{1}{p_\sharp}} \| u\|_{W_0^{1,p_\sharp}(\Omega)}\,d\mu^+(s,p).
\end{aligned}
\]
We observe that~$|\Omega|^{\frac{1}{p}-\frac{1}{p_\sharp}}$ is bounded for~$p\in [p_\sharp,N-\eta]$,
namely
\[
0<\min_{p\in [p_\sharp,N-\eta]}\, |\Omega|^{\frac{1}{p}-\frac{1}{p_\sharp}}\le \max_{p\in [p_\sharp,N-\eta]}\, |\Omega|^{\frac{1}{p}-\frac{1}{p_\sharp}}.
\]
Hence,
\[
\|u\|_\mu\ge\min_{p\in [p_\sharp,N-\eta]}|\Omega|^{\frac{1}{p}-\frac{1}{p_\sharp}}\, \mu^+ \Big(\{1\}\times [p_\sharp,N-\eta]\Big) \| u\|_{W_0^{1,p_\sharp}(\Omega)},
\]
which provides the desired result when~$s_\sharp=1$.

Let us now focus on the case~$s_\sharp\in [\overline s,1)$.
By~\eqref{ipotesisharp}, there exist~$\delta$, $\eta >0$ such that
\[
\mu^+([s_\sharp, 1-\delta]\times [p_\sharp,N-\eta])>0.
\]
Then, since~$p^*_s\geq (p_\sharp)^{*}_{s_\sharp}$ for any~$s\in [s_\sharp, 1-\delta]$ and~$p\in [p_\sharp,N-\eta]$, by the the H\"older inequality and~\cite[Theorem~6.5]{MR2944369} we gather that
\[
\begin{aligned}
\|u\|_\mu&=\iint_\Sigma [u]_{s,p}\,d\mu^+(s,p) \\
&\geq \iint_{[s_\sharp, 1-\delta]\times [p_\sharp,N-\eta]}
 [u]_{s,p}\,d\mu^+(s,p) \\
&\geq \iint_{[s_\sharp, 1-\delta]\times [p_\sharp,N-\eta]} C_1
\|u\|_{L^{p^*_s}(\Omega)}\,d\mu^+(s,p) \\
&\geq \iint_{[s_\sharp, 1-\delta]\times [p_\sharp,N-\eta]} C_2
\|u\|_{L^{(p_\sharp)^{*}_{s_\sharp}}(\Omega)}\,d\mu^+(s,p),
\end{aligned}
\]
being~$C_1=C_1(N,s,p)$ given by~\cite[Theorem~6.5]{MR2944369} and
\[
C_2=C_1(s,p, N, \Omega)\, |\Omega|^{\frac{1}{p^*_s}-\frac{1}{(p_\sharp)^*_{s_\sharp}}}.
\]
We note that the constant~$C_2$ is bounded uniformly for 
any~$s\in [s_\sharp, 1-\delta]$ and~$p\in [p_\sharp,N-\eta]$, namely
\[
0<\min_{(s,p)\in[s_\sharp, 1-\delta]\times [p_\sharp,N-\eta]}
C_2(s,p, N, \Omega)
\le \max_{(s,p)\in[s_\sharp, 1-\delta]\times [p_\sharp,N-\eta]}
C_2(s,p, N, \Omega)<+\infty.
\]
Thus,
\[
\|u\|_\mu \ge\min_{(s,p)\in[s_\sharp, 1-\delta]\times [p_\sharp,N-\eta]} 
C_2(s,p, N, \Omega) \, \mu^+ \Big([s_\sharp, 1-\delta]\times [p_\sharp,N-\eta]\Big) \|u\|_{L^{(p_\sharp)^{*}_{s_\sharp}}(\Omega)},
\]
which provides the desired continuous embedding.

Now, suppose that there exists~$(s_1, p_1)\in (0,1)\times (1,N) $ such that~$\mu^+([s_1,1)\times\{p_1\})>0$.
In this case, 
\[
\mbox{there exists~$\delta>0$ such that } \mu^+([s_1,1-\delta]\times\{p_1\})>0.
\]
Consequently, from Theorem~\ref{sobolevmio}, we have that
\[
\begin{aligned}
\|u\|_\mu&=\iint_\Sigma [u]_{s,p}\,d\mu^+(s,p) \\
&\geq \iint_{[s_1,1-\delta]\times\{p_1\}} [u]_{s,p_1}\,d\mu^+(s,p) \\
&\geq \iint_{[s_1,1-\delta]\times\{p_1\}}C_1 [u]_{s_1,p_1}\,d\mu^+(s,p)
\end{aligned}
\]
for some~$C_1 = C_1(N, \Omega, p_1)>0$.

Thus,
\[
\begin{split}
\|u\|_\mu &\ge 
C_1 \mu^+ \Big([s_1,1-\delta]\times\{p_1\}\Big) [u]_{s_1,p_1}\\
&\ge C_1 \mu^+ \Big([s_1,1-\delta]\times\{p_1\}\Big) \|u\|_{W^{s_1,p_1}_0(\Omega)},
\end{split}
\]
which proves the desired embedding and concludes the proof
of Proposition~\ref{propembedding}.
\end{proof}

\begin{corollary}\label{corembedding}
Let~$s_\sharp$ and~$p_\sharp$ be as in~\eqref{ipotesisharp}. Then,
\begin{itemize}
\item
if~$s_\sharp=1$, then the space~$\mathcal{X}_\mu (\Omega)$ is
compactly embedded in~$L^q(\Omega)$ for any~$q\in [1,{p^*_\sharp})$;
\item
if~$s_\sharp\in [\overline s,1)$ and there exists~$(s_1, p_1)\in (0,1)\times (1,N)$ such that~$\mu^+([s_1,1)\times\{p_1\})>0$, then the space~$\mathcal{X}_\mu (\Omega)$ is compactly embedded in~$L^q(\Omega)$ for any~$q\in [1,(p_1)^*_{s_1})$.
\end{itemize}
\end{corollary}

\begin{proof}
The desired result follows from Proposition~\ref{propembedding}
and the compact embeddings of~$W_0^{1,p_\sharp}(\Omega)$ and~$W_0^{s_1,p_1}(\Omega)$, respectively (see~\cite[Corollary~7.2]{MR2944369}).
\end{proof}

\begin{remark}\label{utile}
We observe that, if~$\mu^+$ reduces to the finite sum of Dirac measures centered at the points~$(s_k,p_k)$, then the assumption that~$\mu^+([s_1,1)\times\{p_1\})>0$ stated in Proposition~\ref{propembedding} and Corollary~\ref{corembedding} is always satisfied.
In particular, in this case the compact embedding is satisfied for any~$q\in [1,(p_\sharp)^*_{s_\sharp})$.
\end{remark}

\subsection{Uniform convexity of the functional space} A natural question in our setting is the following:

\begin{openproblem}\label{OP1} Is the space~$\mathcal X_\mu(\Omega)$ uniformly convex? If not, is it reflexive?
\end{openproblem}

As a partial positive result, here  we reduce ourselves to the case in which~$\mu^+$ is a convergent series of Dirac's delta, namely we assume that there exist a sequence~$(s_k,p_k)\in\Sigma$ and~$c_k\geq  0$ such that
\begin{equation}\label{mu=sommadelta}
\mu^+=\sum_{k=0}^{+\infty}c_k\delta_{(s_k,p_k)}, \quad \mbox{ with } \sum_{k=0}^{+\infty}c_k\in (0,+\infty).
\end{equation}

Under this hypothesis, the following result holds true.

\begin{proposition}\label{uniformeconvesso}
Let~$\mu^+$ satisfy~\eqref{mu=sommadelta}.
Then, $\mathcal{X}_\mu (\Omega)$ is a uniformly convex space.
\end{proposition}

\begin{proof}
{F}rom~\eqref{mu=sommadelta} we have that, for any~$u\in \mathcal{X}_\mu (\Omega)$,
the definition in~\eqref{definizionenorma} reads as
\[
\|u\|_\mu =\sum_{k=0}^{+\infty}c_k[u]_{s_k,p_k}.
\]
We want to prove that for any~$\varepsilon\in (0,2]$ there exists~$\delta>0$ such that,
if~$u$, $v\in \mathcal{X}_\mu (\Omega)$ are such that~$\|u\|_\mu=\|v\|_\mu=1$ and~$\|u-v\|_\mu\geq \varepsilon$, then~$\|u+v\|_\mu\leq 2-\delta$.
To this end, we assume that
\begin{equation}\label{mu=sommadelta1}
\sum_{k=0}^{+\infty}c_k[u]_{s_k,p_k}=\sum_{k=0}^{+\infty}c_k[v]_{s_k,p_k}=1
\end{equation}
and
\begin{equation}\label{mu=sommadelta2}
\sum_{k=0}^{+\infty}c_k[u-v]_{s_k,p_k}\geq \varepsilon.
\end{equation}
{F}rom~\eqref{mu=sommadelta1} we have that, for any~$k\in\N$,
\[
c_k[u]_{s_k,p_k}=[c_k u]_{s_k,p_k}\leq 1\qquad \mbox{and}\qquad c_k[v]_{s_k,p_k}=[c_k v]_{s_k,p_k}\leq 1 
.
\]
Moreover, by~\eqref{mu=sommadelta2}, we infer that there exists some~$\bar{k}$ such that
\[
c_{\bar{k}}[u-v]_{s_{\bar{k}},p_{\bar{k}}}=[c_{\bar{k}} u-c_{\bar{k}} v]_{s_{\bar{k}},p_{\bar{k}}}>0.
\]
Notice that we can define
\[ 
\varepsilon_{\bar{k}}:=c_{\bar{k}}[u-v]_{s_{\bar{k}},p_{\bar{k}}}, 
\]
so that~$\varepsilon_{\bar{k}}\leq  2$.

Now, for this~$\bar{k}$ we have that~$c_{\bar{k}} u$, $c_{\bar{k}} v\in W^{s_{\bar{k}},p_{\bar{k}}}_0(\Omega)$, which is a uniformly convex space with the norm~$[\,\cdot\,]_{s_{\bar{k}},p_{\bar{k}}}$ as~$p_{\bar{k}}>1$.
Thus, there exists some~$\delta_{\bar{k}}>0$ such that
\[
[c_{\bar{k}} u+c_{\bar{k}} v]_{s_{\bar{k}},p_{\bar{k}}}\leq [c_{\bar{k}} u]_{s_{\bar{k}},p_{\bar{k}}}+[c_{\bar{k}} v]_{s_{\bar{k}},p_{\bar{k}}}
-\delta_k.
\]
{F}rom this and the triangle inequality, and recalling~\eqref{mu=sommadelta1}, we deduce that
\[
\|u+v\|_\mu=\sum_{k=0}^{+\infty}c_k[u+v]_{s_k,p_k}\leq \sum_{k=0}^{+\infty}\Big(c_k[u]_{s_k,p_k}+c_k[v]_{s_k,p_k}\Big)-\delta_{\bar{k}} 
=2-\delta_{\bar{k}}.
\]
The desired result follows by setting~$
\delta:=\delta_{\bar{k}}$.
\end{proof}

\subsection{Reabsorbing properties}\label{secmu-}
We now show that the negative components of the signed measure~$\mu$ (if any) can be conveniently ``reabsorbed" into the positive ones.
Indeed, the following result holds:

\begin{theorem}\label{mupiumangiamumeno}
Let~\eqref{ipotesimu11}, \eqref{ipotesimuprodotto}, \eqref{ipotesimus} and~\eqref{ipotesimup} be satisfied.

Then, there exists a constant~$c_0 = c_0(N, \Omega)>0$ such that, for any~$u\in\mathcal X_{\mu}(\Omega)$,
\[
\iint_{[0, \overline s)\times (1, N)} [u]^p_{s, p} \, d\mu^- (s, p) \le c_0 \gamma \iint_{[\overline s,1]\times (1, N)} [u]^p_{s, p} \, d\mu(s, p).
\]
\end{theorem}

\begin{proof}
We notice that, by~\eqref{ipotesimuprodotto} and~\eqref{ipotesimup},
\begin{equation}\label{rimangio1}
\iint_{[0, \overline s)\times (1, N)} [u]^p_{s, p} \, d\mu^- (s, p)=\int_{[1+\delta,N-\eta]}\left(\,\int_{[0, \overline s)}[u]^p_{s, p} \, d\mu_s^-(s)\right)\,d\mu_p(p).
\end{equation}
Moreover, by Theorem~\ref{sobolevmio}, it follows that
\[
\begin{split}
\int_{[1+\delta,N-\eta]}\left(\,\int_{[0, \overline s)}[u]^p_{s, p} \, d\mu_s^-(s)\right)\,d\mu_p (p)&\le
\int_{[1+\delta,N-\eta]}\left(\,\int_{[0, \overline s)}C^p(N,\Omega,p)[u]^p_{\overline s, p} \, d\mu_s^-(s)\right)\,d\mu_p(p)\\
&= \int_{[1+\delta,N-\eta]} C^p(N,\Omega,p)[u]^p_{\overline s, p} \, \mu_s^-\big([0, \overline s)\big)\,d\mu_p(p).
\end{split}
\]
Also, since Theorem~\ref{sobolevmio} holds true for 
every~$p\in (1,N)$, we have that
\begin{equation}\label{costanteteorema1}
C_1(N, \Omega):=\max_{p\in [1+\delta, N-\eta]} C^{p}(N, \Omega, p) <+\infty.
\end{equation}
Therefore, by~\eqref{ipotesimus} and Theorem~\ref{sobolevmio},
\begin{equation}\label{rimangio2}
\begin{split}
\int_{[1+\delta,N-\eta]}\left(\,\int_{[0, \overline s)}[u]^p_{s, p} \, d\mu_s^-(s)\right)\,d\mu_p(p)&\le C_1(N,\Omega)\int_{[1+\delta,N-\eta]}[u]^p_{\overline s, p} \, \mu_s^-\big([0, \overline s)\big)\,d\mu_p(p) \\
&\leq \gamma C_1(N,\Omega)\int_{[1+\delta,N-\eta]}[u]^p_{\overline s, p} \, \mu_s^+\big([\overline s,1]\big)\,d\mu_p(p) \\
&= \gamma C_1(N,\Omega)\int_{[1+\delta,N-\eta]}
\left(\,\int_{[\overline s,1]}[u]^p_{\overline s, p}\,d\mu_s^+(s)\right)\,d\mu_p(p)\\
&\leq \gamma C_1^2(N,\Omega)\int_{[1+\delta,N-\eta]}
\left(\,\int_{[\overline s,1]}[u]^p_{s, p}\,d\mu_s^+(s)\right)\,d\mu_p(p).
\end{split}
\end{equation}
Accordingly, by using~\eqref{rimangio2} and~\eqref{rimangio1},
we conclude that
\[
\begin{aligned}
\iint_{[0, \overline s)\times (1, N)} [u]^p_{s, p} \, d\mu^- (s, p)
&\leq \gamma C_1^2(N,\Omega)\int_{[1+\delta,N-\eta]}
\left(\,\int_{[\overline s,1]}[u]^p_{s, p}\,d\mu_s^+(s)\right)\,d\mu_p(p) \\
&=c_0\gamma\iint_{[\overline s,1]\times (1, N)} [u]^p_{s, p} \, d\mu^+ (s, p),
\end{aligned} 
\]
where~$c_0:=C_1^2(N,\Omega)$. 

Finally, by~\eqref{ipotesimu11}, we see that
\[
\iint_{[\overline s,1]\times (1, N)} [u]^p_{s, p} \, d\mu^+ (s, p)
=\iint_{[\overline s,1]\times (1, N)} [u]^p_{s, p} \, d\mu (s, p),
\]
which concludes the proof.
\end{proof}

{F}rom Theorem~\ref{mupiumangiamumeno} we deduce the following result:

\begin{corollary}\label{corollarymumeno}
Let~\eqref{ipotesimu11}, \eqref{ipotesimuprodotto}, \eqref{ipotesimus} and~\eqref{ipotesimup} be satisfied.

Then, there exists~$c_1 = c_1(N, \Omega)>0$ such that, for any~$u\in\mathcal X_{\mu}(\Omega)$,
\[
\iint_{[0, \overline s)\times (1, N)}\frac1p [u]^p_{s, p} \, d\mu^- (s, p) \le c_1 \gamma \iint_{[\overline s,1]\times (1, N)}\frac1p [u]^p_{s, p} \, d\mu(s, p).
\]
Explicitely, one can take
\begin{equation}\label{constantreabs}
c_1:= N\,C_1^2 (N, \Omega), 
\end{equation}
being~$C_1(N, \Omega)$ as in~\eqref{costanteteorema1}.
\end{corollary}

\begin{proof}
We notice that, since~$p\in (1, N)$, we can employ Theorem~\ref{mupiumangiamumeno}. Thus, there exists a constant~$c_0 = c_0(N, \Omega)>0$ such that
\[
\begin{split}
\iint_{[0, \overline s)\times (1, N)}\frac1p [u]^p_{s, p} \, d\mu^- (s, p) &\le \iint_{[0, \overline s)\times (1, N)} [u]^p_{s, p} \, d\mu^- (s, p)\\
&\le c_0(N, \Omega)  \gamma \iint_{[\overline s,1]\times (1, N)} [u]^p_{s, p} \, d\mu(s, p)\\
&\le c_1(N, \Omega)  \gamma \iint_{[\overline s,1]\times (1, N)}\frac1p [u]^p_{s, p} \, d\mu(s, p),
\end{split}
\]
where~$c_1(N, \Omega):= N c_0(N, \Omega)$.
\end{proof}

\begin{remark}
The ``reabsorbing property" stated in Corollary~\ref{corollarymumeno}
entails that
\[
\iint_{[0, 1]\times (1, N)}\frac1p [u]^p_{s, p} \, d\mu^+(s,p)
-\iint_{[0, \overline s)\times (1, N)}\frac1p [u]^p_{s, p} \, d\mu^-(s,p)
\geq (1-c_1\gamma)\iint_{[0, 1]\times (1, N)}\frac1p [u]^p_{s, p} \, d\mu^+(s,p).
\]
Thus, if~$\gamma$ in~\eqref{ipotesimus} is small enough, namely
\begin{equation}\label{gammabound}
\gamma<\frac{1}{c_1},\quad
\mbox{being~$c_1$ as in~\eqref{constantreabs}},
\end{equation}
we have that
\begin{equation}\label{211BIS}
\iint_{[0, 1]\times (1, N)}\frac1p [u]^p_{s, p} \, d\mu^+(s,p)
-\iint_{[0, \overline s)\times (1, N)}\frac1p [u]^p_{s, p} \, d\mu^-(s,p)>
c\iint_{[0, 1]\times (1, N)}\frac1p [u]^p_{s, p} \, d\mu^+(s,p),
\end{equation}
for some~$c\in(0,1)$.
\end{remark}

\section{Existence of solution for problems~\eqref{problemaW} and~\eqref{senzamumeno}}\label{sec3}

In this section we give the proofs of Theorems~\ref{Weierstrass} 
and~\ref{MPT}.

\subsection{Preliminary observations}
We prove here a useful integration by parts formula.

\begin{proposition}\label{prop:intebyparts}
Let~$u$, $v\in C^2(\Omega)$ be such that~$u\equiv v\equiv0$ in~$\R^n\setminus\Omega$. Assume that
\[
v\,(-\Delta)^s_p u\in L^1(\Omega\times\Sigma, dx\, d\mu(s, p))
\]
and
\[
(x, y, s, p)\mapsto c_{N, s, p}\frac{|u(x)-u(y)|^{p-2} (u(x)-u(y))(v(x)-v(y))}{|x-y|^{N+sp}}\in L^1\big(\R^N\times\Sigma, dx\,dy\, d\mu(s, p)\big).
\]

Then,
\begin{equation}\label{intsp}
\begin{split}
&\iint_\Sigma \frac{c_{N, s, p}}{2} \iint_{\R^{2N}} \frac{|u(x)-u(y)|^{p-2} (u(x)-u(y))(v(x)-v(y))}{|x-y|^{N+sp}}\, dx \,dy\, d\mu(s, p)\\
&\qquad= \iint_\Sigma \int_\Omega v(x) (-\Delta)^s_pu(x) \, dx\, d\mu(s, p).
\end{split}
\end{equation}
\end{proposition}

\begin{proof}
We point out that the integrability assumptions in Proposition~\ref{prop:intebyparts} guarantee that
both sides of~\eqref{intsp} are finite.

We aim to show that, for any~$(s, p)\in\Sigma$, 
\[
\frac{c_{N, s, p}}{2} \iint_{\R^{2N}} \frac{|u(x)-u(y)|^{p-2} (u(x)-u(y))(v(x)-v(y))}{|x-y|^{N+sp}}\, dx\, dy = \int_\Omega v(x) (-\Delta)^s_p u(x) \, dx.
\]
{F}rom this, the identity in~\eqref{intsp} will plainly follow.

Now, we have that
\begin{eqnarray*}
&&c_{N, s, p}  \iint_{\R^{2N}} \frac{|u(x)-u(y)|^{p-2} (u(x)-u(y))(v(x)-v(y))}{|x-y|^{N+sp}}\, dx\, dy \\
&=&c_{N, s, p}\iint_{\R^{2N}} v(x)\frac{|u(x)-u(y)|^{p-2} (u(x)-u(y))}{|x-y|^{N+sp}}\, dx\, dy \\
&&\qquad-c_{N, s, p}\iint_{\R^{2N}} v(y)\frac{|u(x)-u(y)|^{p-2} (u(x)-u(y))}{|x-y|^{N+sp}}\, dx\, dy \\
&=&2c_{N, s, p}\int_{\R^N} v(x)\int_{\R^N} \frac{|u(x)-u(y)|^{p-2} (u(x)-u(y))}{|x-y|^{N+sp}}\, dx\, dy \\
&= &2\int_\Omega v(x) \, (-\Delta)^s_p u(x) \,dx,
\end{eqnarray*}
as desired.
\end{proof}


\subsection{Proof of Theorem~\ref{Weierstrass}}
In this subsection we prove the existence of a nontrivial solution for problem~\eqref{problemaW}.
Throughout this subsection, we let~$\mu$ be as in~\eqref{muvbdshail} and we assume that~\eqref{anchequestat4390}, \eqref{ipotesimus} and~\eqref{ipotesimup} are satisfied.

We recall that, in what follows, $\overline s=\max\{s_1,\dots,s_n\}$.
Moreover, the exponents~$s_\sharp$ and~$p_\sharp$ in~\eqref{ipotesisharp} in this case boil down to~$\overline s$ and~$\widehat p$.

We remark that, in this setting,
\begin{equation}\label{musommafinita}
\mu^+(s,p):= \sum_{{i=1,\dots,n}\atop{j=1,\dots,m }}\delta_{(s_i, p_j)}.
\end{equation}
Notice that since~$\mu^+$ is a finite sum of Dirac's deltas, it can be rewritten
in the form of~\eqref{mu=sommadelta}
(and therefore, in particular, Proposition~\ref{uniformeconvesso}
holds true in this case).

With this choice,  the definition in~\eqref{definizionenorma} boils down to
\begin{equation}\label{nnorm}
\|u\|_\mu = \sum_{{i=1,\dots,n}\atop{j=1,\dots,m }} [u]_{s_i, p_j}.
\end{equation}
Moreover, taking into account~\eqref{muvbdshail},
we recall that the operator defined in~\eqref{elle} becomes
\begin{equation}\label{ellefinito}
L_\mu \, u:= \sum_{{i=1,\dots,n}\atop{j=1,\dots,m }}  (-\Delta)^{s_i}_{p_j} u - \sum_{j=1}^m \int_{[0, 1]} (-\Delta)^s_{p_j} u \, d\mu_s^-(s).
\end{equation}

In this setting, we have the following
definition of weak solutions for~\eqref{problemaW}.

\begin{definition}\label{defnwW}
We say that~$u\in\mathcal{X}_\mu (\Omega)$ is a weak solution of the problem~\eqref{problemaW} if, for any~$v\in\mathcal{X}_\mu (\Omega)$,
\[
\begin{split}
&\sum_{{i=1,\dots,n}\atop{j=1,\dots,m }}  \frac{c_{N, s_i, p_j}}{2} \iint_{\R^{2N}} \frac{|u(x)-u(y)|^{p_j-2} (u(x)-u(y))(v(x)-v(y))}{|x-y|^{N+s_i p_j}}\, dx\, dy \\
&\quad-\sum_{j=1}^m\int_{[0, \overline s)}\frac{c_{N, s, p_j}}{2} \iint_{\R^{2N}} \frac{|u(x)-u(y)|^{p_j-2} (u(x)-u(y))(v(x)-v(y))}{|x-y|^{N+sp_j}}\, dx\, dy\, d\mu_s^-(s)\\
&\qquad = \int_\Omega g(x) v(x)\, dx.
\end{split}
\]
\end{definition}

\begin{remark}
We stress that expressions such as
\[
\sum_{{i=1,\dots,n}\atop{j=1,\dots,m }} \frac{c_{N, s_i, p_j}}{2} \iint_{\R^{2N}} \frac{|u(x)-u(y)|^{p_j-2} (u(x)-u(y))(v(x)-v(y))}{|x-y|^{N+s_i p_j}}\, dx\, dy
\]
and
\[
\sum_{j=1}^m\int_{[0, \overline s)}\frac{c_{N, s, p_j}}{2} \iint_{\R^{2N}} \frac{|u(x)-u(y)|^{p_j-2} (u(x)-u(y))(v(x)-v(y))}{|x-y|^{N+sp_j}}\, dx\, dy\, d\mu_s^-(s)
\]
constitute a slight abuse of notation. To be precise,
for any~$i$ such that~$s_i=1$, one should write instead
\[
\sum_{j=1}^m \frac{1}{p_j}\int_\Omega |\nabla u(x)|^{p_j -2} \nabla u(x) \cdot\nabla v(x) \,dx.
\]
Similarly, for any~$i$ such that~$s_i=0$, one should write
\[
\sum_{j=1}^m \frac{1}{p_j}\int_\Omega |u(x)|^{p_j -2} u (x)v(x) \,dx.
\]
For the sake of shortness, however, we will accept the above abuse of notation whenever typographically convenient.
\end{remark}

Now, let~$g\in L^{\widehat{p}'}(\Omega)$.
Let~$J:\mathcal{X}_\mu (\Omega)\to\R$ be the functional defined as
\begin{equation}\label{functionalnJ}
\begin{split}
J(u):= \sum_{{i=1,\dots,n}\atop{j=1,\dots,m }}  [u]^{p_j}_{s_i, p_j} -\sum_{j=1}^m\int_{[0, \overline s)} \frac{1}{2p_j} [u]_{s, p_j}^{p_j} \,d\mu_s^-(s) -\int_\Omega g(x)u \,dx.
\end{split}
\end{equation}
Then, we have the following properties of~$J$:

\begin{lemma}
For every~$u\in\mathcal{X}_\mu (\Omega)$, we have that
\[
|J(u)|<+\infty.
\]
\end{lemma}

\begin{proof}
Let~$u\in\mathcal{X}_\mu (\Omega)$.
Recalling~\eqref{nnorm}, we gave that
$$\sum_{{i=1,\dots,n}\atop{j=1,\dots,m }} [u]_{s_i, p_j}=\|u\|_\mu <+\infty.$$
This entails that
\[
\sum_{{i=1,\dots,n}\atop{j=1,\dots,m }}  \frac{1}{2p_j} [u]^{p_j}_{s_i, p_j} <+\infty.
\]

Moreover, by~\eqref{musommafinita} and Corollary~\ref{corollarymumeno}, we infer that
\[
\sum_{j=1}^m\int_{[0, \overline s)} \frac{1}{2p_j} [u]_{s, p_j}^{p_j} \,d\mu_s^-(s)\le c_1 \gamma \sum_{{i=1,\dots,n}\atop{j=1,\dots,m }}  \frac{1}{2p_j} [u]^{p_j}_{s_i, p_j}<+\infty.
\]
Also, by the H\"older inequality (with exponents~$\widehat{p}$
and~$\widehat{p}'$)
and Proposition~\ref{propembedding} (recall that~$p_\sharp=\widehat p$ in this setting), we have that
\begin{eqnarray*}
&&\left|\int_\Omega g(x)u(x)\, dx\right|\le
\left(\int_{\Omega}|g(x)|^{\widehat{p}'}\,dx\right)^{\frac1{\widehat{p}'}}
\left(\int_{\Omega}|u(x)|^{\widehat{p}}\,dx\right)^{\frac1{\widehat{p}}}
<+\infty
.\end{eqnarray*}
Gathering these pieces of information, we obtain that~$
|J(u)|<+\infty$, as desired.
\end{proof}

\begin{proposition}\label{Jbendefinito}
We have that, for all~$u$, $v\in\mathcal{X}_\mu (\Omega)$,
\begin{equation}\label{diff1}
\begin{split}
\langle J'(u), v\rangle &= \sum_{{i=1,\dots,n}\atop{j=1,\dots,m }} \frac{c_{N, s_i, p_j}}{2} \iint_{\R^{2N}} \frac{|u(x)-u(y)|^{p_j-2} (u(x)-u(y))(v(x)-v(y))}{|x-y|^{N+s_i p_j}}\, dx\, dy\\
&-\sum_{j=1}^m\int_{[0, \overline s)}\frac{c_{N, s, p_j}}{2} \iint_{\R^{2N}} \frac{|u(x)-u(y)|^{p_j-2} (u(x)-u(y))(v(x)-v(y))}{|x-y|^{N+sp_j}}\, dx\, dy\, d\mu_s^-(s)\\
&-\int_\Omega g(x) v (x)\,dx.
\end{split}
\end{equation}

Furthermore, critical points of~$J$ are weak solutions of the problem~\eqref{problemaW}.
\end{proposition}

\begin{proof}
By~\eqref{functionalnJ}, we obtain that
\[
\begin{split}
&J(u+\varepsilon v) = J(u) + \varepsilon \sum_{{i=1,\dots,n}\atop{j=1,\dots,m }}  \frac{c_{N, s_i, p_j}}{2} \iint_{\R^{2N}} \frac{|u(x)-u(y)|^{p_j-2} (u(x)-u(y))(v(x)-v(y))}{|x-y|^{N+s_i p_j}}\, dx\, dy\\
&\quad-\varepsilon\sum_{j=1}^m\int_{[0, \overline s)}\frac{c_{N, s, p_j}}{2} \iint_{\R^{2N}} \frac{|u(x)-u(y)|^{p_j-2} (u(x)-u(y))(v(x)-v(y))}{|x-y|^{N+sp_j}}\, dx\, dy\, d\mu_s^-(s)\\
&\quad+ o(\varepsilon)  \left(\sum_{{i=1,\dots,n}\atop{j=1,\dots,m }}  \frac{1}{2p_j} [u]^{p_j}_{s_i, p_j} -\sum_{j=1}^m\int_{[0, \overline s)} \frac{1}{2p_j} [u]_{s, p_j}^{p_j} \,d\mu_s^-(s)\right) +\varepsilon \int_\Omega g(x)v(x) \,dx.
\end{split}
\]
{F}rom this, we infer that
\[
\begin{split}
&\lim_{\varepsilon\to 0}\frac{J(u+\varepsilon v)-J(u)}{\varepsilon} =  \sum_{{i=1,\dots,n}\atop{j=1,\dots,m }} \frac{c_{N, s_i, p_j}}{2} \iint_{\R^{2N}} \frac{|u(x)-u(y)|^{p_j-2} (u(x)-u(y))(v(x)-v(y))}{|x-y|^{N+s_i p_j}}\, dx\, dy\\
&\qquad-\sum_{j=1}^m\int_{[0, \overline s)}\frac{c_{N, s, p_j}}{2} \iint_{\R^{2N}} \frac{|u(x)-u(y)|^{p_j-2} (u(x)-u(y))(v(x)-v(y))}{|x-y|^{N+sp_j}}\, dx\, dy\, d\mu_s^-(s)\\
&\qquad-\int_\Omega g(x) v (x)\,dx,
\end{split}
\]
namely~\eqref{diff1} holds true.

{F}rom~\eqref{diff1}, we see that if~$u$ is a critical point of~$J$, then~$u$ is a weak solution of~\eqref{problemaW}, thanks to Definition~\ref{defnwW}.
\end{proof}

\begin{lemma}\label{u3itrfgekauqgo3wutiy348o87654}
The functional~$J$ is weakly lower semicontinuous.
\end{lemma}

\begin{proof}
We let~$u_k$ be a sequence in~$\mathcal X_\mu(\Omega)$ which converges
weakly to some~$u\in \mathcal X_\mu(\Omega)$ as~$k\to+\infty$
and we claim that
\begin{equation}\label{chevuoldirewlsc}
\liminf_{k\to +\infty}J(u_k)\geq J(u).
\end{equation}
To check this, we observe that, by Fatou's Lemma,
\begin{equation}\label{chevuoldirewlsc1}
\liminf_{k\to +\infty}\sum_{{i=1,\dots,n}\atop{j=1,\dots,m }} \frac{1}{2p_j} [u_k]^{p_j}_{s_i, p_j} \ge\sum_{{i=1,\dots,n}\atop{j=1,\dots,m }} \frac{1}{2p_j} [u]^{p_j}_{s_i, p_j}.
\end{equation}

Moreover, by Corollary~\ref{corembedding} we infer that~$u_k$ 
converges strongly to~$u$ in~$L^{\widehat{p}}(\Omega)$, and thus,
by the H\"older inequality,
\begin{eqnarray*}
&&\lim_{k\to+\infty}
\left| \int_\Omega g(x)\big(u_k(x)-u(x)\big) \,dx\right|
\le\lim_{k\to+\infty} \|g\|_{L^{\widehat{p}'}(\Omega)}\|u_k-u\|_{L^{\widehat{p}}(\Omega)}=0,
\end{eqnarray*}
which implies that
\begin{equation}\label{chevuoldirewlsc2}
\liminf_{k\to +\infty} \int_\Omega g(x)u_k(x) \,dx =
\int_\Omega g(x)u(x) \,dx.
\end{equation}

Now, we want to show that
\begin{equation}\label{chevuoldirewlsc3}
\lim_{k\to +\infty}\sum_{j=1}^m\int_{[0, \overline s)} \frac{1}{2p_j} [u_k]_{s, p_j}^{p_j} \,d\mu_s^-(s)
=\sum_{j=1}^m\int_{[0, \overline s)} \frac{1}{2p_j} [u]_{s, p_j}^{p_j} \,d\mu_s^-(s).
\end{equation}
Since~$u_k$ converges weakly in~$\mathcal X_\mu(\Omega)$, it is 
bounded in~$\mathcal X_\mu(\Omega)$.
Thus, for any~$i\in\{1,\dots,n\}$ and~$j\in\{1,\dots,m\}$, $u_k$ is bounded in~$W^{s_i,p_j}(\R^N)$. 
Up to a subsequence, we can assume that~$u_k$ converges weakly to~$u$
in~$W^{s_i,p_j}(\R^N)$ for any~$i$ and~$j$.
Then, from Corollary~\ref{corembedding} we have that
\begin{equation}\label{convlpj}
u_k\to u \ \mbox{as~$k\to+\infty$ in $L^{p_j}(\Omega)$, for any $j$}. 
\end{equation}

Moreover, for any~$s\in[0,\overline{s})$, we can exploit~\cite[Theorem 1]{MR3813967} and obtain that, for any~$\theta\in(0, 1)$, there exists a positive constant~$C=C(N, \theta, p_j, \overline s)$ such that
\[
\|u_k-u\|_{W^{s,p_j}(\R^N)}\le 
C\|u_k-u\|_{L^{p_j}(\R^N)}^\theta
\|u_k-u\|_{W^{\overline s,p_j}(\R^N)}^{1-\theta}.
\]
This and~\eqref{convlpj} entail that, for all~$j\in\{1,\dots,m\}$,
\begin{eqnarray*}&&
\lim_{k\to +\infty}\int_{[0, \overline s)} [u_k-u]_{s, p_j}^{p_j} \,d\mu_s^-(s)\le 
\lim_{k\to +\infty}\int_{[0, \overline s)} \|u_k-u\|_{W^{s,p_j}(\R^N)}^{p_j} \,d\mu_s^-(s)\\&&\qquad\qquad\le C
\lim_{k\to +\infty}\int_{[0, \overline s)} \|u_k-u\|_{L^{p_j}(\R^N)}^{\theta p_j}
\|u_k-u\|_{W^{\overline s,p_j}(\R^N)}^{(1-\theta)p_j} \,d\mu_s^-(s)\\
&&\qquad\qquad\le C
\mu_s^-([0, \overline s))\lim_{k\to +\infty}
\|u_k-u\|_{L^{p_j}(\R^N)}^{\theta p_j}\\&&\qquad\qquad=0.
\end{eqnarray*}
This gives~\eqref{chevuoldirewlsc3}.

Combining~\eqref{chevuoldirewlsc1}, \eqref{chevuoldirewlsc2} and~\eqref{chevuoldirewlsc3} we get the claim in~\eqref{chevuoldirewlsc}.
\end{proof}

We recall that, for every~$x_1,\dots,x_m\in \R$, 
\begin{equation}\label{elementaryinequality}
|x_1+\cdots+x_m|^p\leq m^{p-1}(|x_1|^p+\cdots +|x_m|^p)\quad\mbox{for any } p>1.
\end{equation}

We also point out that so far we have not used the smallness
of the parameter~$\gamma$ appearing in the assumption~\eqref{ipotesimus}. The next statement instead will
require the additional assumption that~$\gamma$ is sufficiently small in order
to reabsorb the contribution coming from the negative part of the measure.

\begin{lemma}\label{u3itrfgekauqgo3wutiy348o87654BIS}
There exists~$\gamma_0>0$ such that, if~$\gamma\in[0,\gamma_0]$,
the functional~$J$ is coercive, namely
\begin{equation}\label{claimcoercivo}
\lim_{\|u\|_{\mu}\to+\infty}\frac{J(u)}{\|u\|_{\mu}}=+\infty.
\end{equation}
\end{lemma}

\begin{proof}
{F}rom~\eqref{211BIS}, \eqref{functionalnJ}
and the H\"older inequality we have that
\[
J(u)\ge c\sum_{{i=1,\dots,n}\atop{j=1,\dots,m}} \frac{1}{2p_j} [u]^{p_j}_{s_i, p_j} -\|g\|_{L^{\widehat{p}'}(\Omega)}\|u\|_{L^{\widehat{p}}(\Omega)},
\]
provided that~$\gamma$ is taken sufficiently small.

This and Proposition~\ref{propembedding} give that
\begin{equation}\label{314BIS}
J(u)\ge c\sum_{{i=1,\dots,n}\atop{j=1,\dots,m}} \frac{1}{2p_j} [u]^{p_j}_{s_i, p_j}  -C\|u\|_{\mu},
\end{equation}
for some~$C>0$.

Now we notice that if~$\|u\|_\mu \ge nm$,
then the set
\[
K:=\Big\{(i, j)\in\{1,\cdots,n\}\times\{1,\cdots,m\} \mbox{ s.t. } [u]_{s_i,p_j}\ge 1 \Big\}
\]
is non empty. 

Hence, if~$\|u\|_{\mu}$ goes to infinity, 
then, for at least one couple~$(i, j)\in K$, we have that~$[u]_{s_i,p_j}$ goes to infinity. Accordingly,
\begin{equation}\label{sommaKdivergeJ}
\lim_{\|u\|_{\mu}\to+\infty}\sum_{(i, j)\in K} [u]_{s_i,p_j}=+\infty.
\end{equation}

We denote by~$|K|$ the cardinality of the set~$K$ and we define
\[
p_{min}:=\min_{(i, j)\in K}\{p_j\}.
\]
As a result, by using~\eqref{elementaryinequality}, 
\begin{equation}\label{sommaKdiverge2J}
\sum_{{i=1,\dots,n}\atop{j=1,\dots,m}}[u]_{s_i,p_j}^{p_j}\ge \sum_{(i, j)\in K}[u]_{s_i,p_j}^{p_j}\ge \sum_{(i, j)\in K}[u]_{s_i,p_j}^{p_{min}}\ge \frac{1}{|K|^{p_{min}-1}}
\left(\sum_{(i, j)\in K}[u]_{s_i,p_j}\right)^{p_{min}}.
\end{equation}
We also have that
\begin{equation*}
\|u\|_\mu=\sum_{{i=1,\dots,n}\atop{j=1,\dots,m}} [u]_{s_i,p_j}
\sum_{(i, j)\in K}[u]_{s_i,p_j}+\sum_{(i, j)\in \{1,\cdots,n\}\times\{1,\cdots,m\}\setminus K}[u]_{s_i,p_j}
\le \sum_{(i, j)\in K}[u]_{s_i,p_j}+nm-|K|.
\end{equation*}
Therefore, combining this, \eqref{314BIS}, \eqref{sommaKdivergeJ}
and~\eqref{sommaKdiverge2J},
\[
\begin{aligned}
\lim_{\|u\|_{\mu}\to+\infty}\frac{J(u)}{\|u\|_\mu}&
\ge \lim_{\|u\|_{\mu}\to+\infty} \frac{c}{2\widehat{p}\,\|u\|_\mu}
\sum_{{i=1,\dots,n}\atop{j=1,\dots,m}} [u]_{s_i,p_j}^{p_j} -C \\
&\geq \frac{c}{2\widehat{p}\,|K|^{p_{min}-1}} \lim_{\|u\|_{\mu}\to+\infty}
\dfrac{\displaystyle \left(\sum_{(i, j)\in K}[u]_{s_i,p_j}\right)^{p_{min}}}{\displaystyle\sum_{(i, j)\in K}[u]_{s_i,p_j}+nm-|K|}-C\\
&=+\infty,
\end{aligned}
\]
which proves the claim in~\eqref{claimcoercivo}.
\end{proof}

We now proceed with the proof of Theorem~\ref{Weierstrass}.

\begin{proof}[Proof of Theorem~\ref{Weierstrass}]
In light of Proposition~\ref{Jbendefinito}, we aim at finding
critical points of the functional~$J$ in~\eqref{functionalnJ}.

Since~$J$ is weakly lower semicontinuous (thanks to Lemma~\ref{u3itrfgekauqgo3wutiy348o87654})
and
coercive (thanks to Lemma~\ref{u3itrfgekauqgo3wutiy348o87654BIS})
and the space~$\mathcal X_\mu(\Omega)$ is reflexive (see Proposition~\ref{uniformeconvesso}), by the Weierstrass theorem, we infer the existence of a minimizer for the functional~$J$,
which is a weak solution of problem~\eqref{problemaW}.
This establishes the first claim in Theorem~\ref{Weierstrass}.

In addition, if~$\mu^-\equiv 0$, then the functional~$J$ reduces to
\[
J(u)= \sum_{{i=1,\dots,n}\atop{j=1,\dots,m }}
\frac{1}{2p_j} [u]^{p_j}_{s_i, p_j} -\int_\Omega g(x)u (x)\,dx.
\]
Since~$J$ is strictly convex, we see that
the global minimum is the only critical point of~$J$, concluding the proof of Theorem~\ref{Weierstrass}.
\end{proof}

\subsection{Proof of the Theorem~\ref{MPT}}
We now focus on giving the proof of Theorem~\ref{MPT}.
We recall that, in this setting,
\begin{equation}\label{musommafinitakappa}
\mu^+(s,p)= \sum_{k=1}^m \delta_{(s_k, p_k)},
\end{equation}
namely~$\mu^+$ is a particular case of the measure defined in~\eqref{mu=sommadelta}. With this choice,  the definition in~\eqref{definizionenorma} boils down to
\begin{equation}\label{nnormkappa}
\|u\|_\mu =\sum_{k=1}^m [u]_{s_k, p_k}
\end{equation}
and the operator defined in~\eqref{elle} becomes
\begin{equation}\label{ellefinito}
L_\mu \, u:= \sum_{k=1}^m (-\Delta)^{s_k}_{p_k} u.
\end{equation}

Our goal is to show the existence of a Mountain Pass solution for the problem~\eqref{senzamumeno}, under
the assumptions that~$f:\Omega\times\R\to\R$ is a Carth\'eodory function satisfying~\eqref{AR1}, \eqref{AR2}, \eqref{AR3} and~\eqref{AR4}.

With this aim, we provide the definition of weak solutions for the problem~\eqref{senzamumeno}.

\begin{definition}\label{defnws}
We say that~$u\in\mathcal{X}_\mu (\Omega)$ is a weak solution of~\eqref{senzamumeno} if, for any~$v\in\mathcal{X}_\mu (\Omega)$,
\[
\begin{split}
&\sum_{k=1}^m \frac{c_{N, s_k, p_k}}{2} \iint_{\R^{2N}} \frac{|u(x)-u(y)|^{p_k-2} (u(x)-u(y))(v(x)-v(y))}{|x-y|^{N+s_k p_k}}\, dx\, dy \\
&\qquad\qquad\qquad = \int_\Omega f(x, u(x)) v(x)\, dx.
\end{split}
\]
\end{definition}

Also, let~$F$ be defined as in~\eqref{definizioneF} and~$I:\mathcal{X}_\mu (\Omega)\to\R$ be the functional defined as
\begin{equation}\label{functionaln}
\begin{split}
I(u):= \sum_{k=1}^m \frac{1}{2p_k} [u]^{p_k}_{s_k, p_k} -\int_\Omega F(x, u(x)) \,dx.
\end{split}
\end{equation}

Bu arguing as in the proof of Proposition~\ref{Jbendefinito}, we
obtain the following statement:

\begin{proposition}
We have that, for all~$u$, $v\in\mathcal{X}_\mu (\Omega)$,
\begin{equation}\label{derfini7952qlwhfqor3u2t8y}\begin{split}
\langle I'(u), v\rangle =&
\sum_{k=1}^m\frac{c_{N, s_k, p_k}}{2} \iint_{\R^{2N}} \frac{|u(x)-u(y)|^{p_k-2} (u(x)-u(y))(v(x)-v(y))}{|x-y|^{N+s_k p_k}}\, dx\, dy\\
&\qquad\qquad-\int_\Omega  f(x, u(x)) v(x)\,dx.\end{split}
\end{equation}

Furthermore, critical points of~$I$ are weak solutions of the problem~\eqref{senzamumeno}.
\end{proposition}

We also recall the following result stated in~\cite[Lemma~3]{SV12}:

\begin{lemma}\label{SV12}
Let~$f$ satisfy~\eqref{AR1} and~\eqref{AR2}. Then, for any~$\varepsilon>0$
there exists~$\delta=\delta(\varepsilon)>0$ such that,
for a.e.~$x\in \Omega$ and for any~$t\in \R$,
\begin{equation}\label{F1}\begin{split}
&|f(x,t)|\leq \widehat{p}\varepsilon|t|^{\widehat{p}-1}+q\delta(\varepsilon)|t|^{q-1}
\\
{\mbox{and }}\qquad&
|F(x,t)|\leq \varepsilon|t|^{\widehat{p}}+\delta(\varepsilon)|t|^q,
\end{split}\end{equation}
where~$F$ is defined in~\eqref{definizioneF}.
\end{lemma}

We now prove that the geometry of the Mountain Pass Theorem is satisfied.

\begin{proposition}\label{geo1}
Let~$f$ satisfy~\eqref{AR1} and~\eqref{AR2}.

Then, there exist~$\rho$, $\beta>0$ such that, for any~$u\in \mathcal X_\mu(\Omega)$ with~$\|u\|_\mu=\rho$,  we have that~$I(u)\geq \beta$.
\end{proposition}

\begin{proof}
By~\eqref{musommafinitakappa} and Remark~\ref{utile}, we have that~$\mathcal X_\mu(\Omega)$ is compactly embedded in~$L^r(\Omega)$ for any~$r\in [1, (p_\sharp)^*_{s_\sharp})$, and therefore it is compactly embedded in~$L^{\widehat p}(\Omega)$
and in~$L^q(\Omega)$
(recall that~$q\in (\widehat{p}, (p_\sharp)^*_{s_\sharp})$, in light
of the assumption in~\eqref{AR1}).

Hence, by means of~\eqref{functionaln}, \eqref{F1} and~\eqref{nnormkappa}, we infer that, for any~$\varepsilon>0$,
there exists~$\delta(\varepsilon)>0$ such that, for all~$u\in\mathcal{X}_\mu (\Omega)$,
\begin{eqnarray*}
I(u) &\ge& \sum_{k=1}^m \frac{1}{2p_k} [u]^{p_k}_{s_k, p_k} -\varepsilon\|u\|^{\widehat p}_{L^{\widehat p}(\Omega)} - \delta(\varepsilon)\|u\|^q_{L^q(\Omega)}\\
&\ge& \frac{1}{2\widehat p} \sum_{k=1}^m [u]^{p_k}_{s_k, p_k} -C\varepsilon\|u\|^{\widehat p}_{\mu} - C\delta(\varepsilon)\|u\|^q_{\mu}
,
\end{eqnarray*}
for some~$C>0$.
 
{F}rom this, and using also~\eqref{elementaryinequality},
we see that, if~$\|u\|_\mu\le 1$, 
\[
\begin{split}
I(u)&\ge\frac{1}{2\widehat p}\sum_{k=1}^m  [u]^{\widehat p}_{s_k, p_k}-C\varepsilon\|u\|^{\widehat p}_{\mu} - C\delta(\varepsilon)\|u\|^q_{\mu}\\
&\ge \frac{1}{2 \widehat p\, m^{\widehat p -1}}
\left(\sum_{k=1}^m  [u]_{s_k, p_k}\right)^{\widehat p} - 
C\varepsilon \|u\|_\mu^{\widehat p}-C\delta(\varepsilon)\|u\|^q_{\mu} \\
&=\|u\|^{\widehat p}_\mu \left(\frac{1}{2\widehat p \,m^{\widehat p -1}} - C\varepsilon-C\delta(\varepsilon) \|u\|^{q-\widehat p}_\mu\right).
\end{split}
\]
Thus, choosing
$$ \varepsilon:=\frac{1}{4C\widehat p \,m^{\widehat p -1}},
$$
we obtain that
$$ I(u)\ge \|u\|^{\widehat p}_\mu \left(\frac{1}{4\widehat p \,m^{\widehat p -1}} - C\delta \|u\|^{q-\widehat p}_\mu\right),$$
where now~$\delta$ depends on the structural constants of the problem.

Now we take
$$ \rho:=\min\left\{ \frac{1}{8C\delta\widehat p \,m^{\widehat p -1}},1\right\}.
$$
In this way, if~$\|u\|_\mu=\rho$,
$$ I(u)\ge \rho^{\widehat p}\frac{1}{8\widehat p \,m^{\widehat p -1}}=:\beta >0,$$
as desired.
\end{proof}

\begin{proposition}\label{geo2}
Let~$f$ satisfy~\eqref{AR1}, \eqref{AR2} and~\eqref{AR4}.

Then, there exists~$e\in \mathcal X_\mu(\Omega)$ such that~$e\geq 0$ a.e. in~$\R^N$, $\|e\|_\mu>\rho$ and~$I(e)<\beta$, where~$\rho$ and~$\beta$ are given in Proposition~\ref{geo1}.
\end{proposition}

\begin{proof}
Let~$t\ge1$ and let~$\varphi\in C^\infty_0(\Omega)$ with~$\varphi\ge0$ in~$\R^N$ and~$\varphi=1$ in some ball~$B\Subset\Omega$.

By~\eqref{functionaln}, \eqref{nnormkappa} and~\eqref{AR4}, we infer that
\[
\begin{split}
I(t\varphi)&= \sum_{k=1}^n \frac{t^{p_k}}{2p_k} [\varphi]^{p_k}_{s_k, p_k} -\int_\Omega F(x, t\varphi(x)) \,dx\\
&\le \sum_{k=1}^m \frac{t^{p_k}}{2p_k} [\varphi]^{p_k}_{s_k, p_k}-a_3t^{\widetilde{\vartheta}}\|\varphi\|_{L^{\widetilde{\vartheta}}(\Omega)}^{\widetilde{\vartheta}} +\int_\Omega a_4(x)\,dx\\
&\le\sum_{k=1}^m \frac{t^{\widehat p}}{2p_k} [\varphi]^{p_k}_{s_k, p_k}-a_3t^{\widetilde{\vartheta}}\|\varphi\|_{L^{\widetilde{\vartheta}}(\Omega)}^{\widetilde{\vartheta}} +\int_\Omega a_4(x)\,dx.
\end{split}
\]

Let also
$$ \bar t:=\max\left\{\rho,\left(\frac{\sum_{k=1}^m \frac{1}{p_k} [\varphi]^{p_k}_{s_k, p_k}}{
a_3\|\varphi\|_{L^{\widetilde{\vartheta}}(\Omega)}^{\widetilde{\vartheta}}
}\right)^{\frac1{\widetilde\vartheta-\widehat p}},\left(\frac{
\int_\Omega a_4(x)\,dx}{
a_3\|\varphi\|_{L^{\widetilde{\vartheta}}(\Omega)}^{\widetilde{\vartheta}}
}
\right)^{\frac1{\widetilde\vartheta}} \right\}+1.
$$
In this way, since~$\widetilde{\vartheta}>\widehat p$,
taking~$e:=\frac{\bar t\varphi}{\|\varphi\|_\mu}$, we see that
$$ \|e\|_\mu>\rho\qquad {\mbox{and}}\qquad I(e)<\beta,$$
as desired.\end{proof}

Now, we want to prove that the Palais-Smale condition holds at any level. We start by proving that every Palais-Smale sequence is bounded.

\begin{proposition}\label{PropPS1}
Let~$f$ satisfy~\eqref{AR1}, \eqref{AR2} and~\eqref{AR3}. Let~$c\in\R$ and~$u_n$ be a sequence in~$\mathcal{X}_\mu(\Omega)$ such that
\begin{equation}\label{bdd1}
\lim_{n\to +\infty} I(u_n)=c
\end{equation}
and
\begin{equation}\label{bdd2}
\lim_{n\to +\infty} \sup_{\substack{v\in \mathcal{X}_{ \mu}(\Omega)\\ \|v\|_{ \mu}=1}} |\langle I'(u_n), v\rangle| =0.
\end{equation}

Then, $u_n$ is bounded in~$\mathcal{X}_{\mu}(\Omega)$.
\end{proposition}

\begin{proof}
By~\eqref{bdd1} and~\eqref{bdd2}, there exists~$M>0$ such that, for any~$n\in\N$,
\begin{equation*}
|I(u_n)|\le M\qquad\mbox{and}\qquad \left|\left\langle I'(u_n), \frac{u_n}{\|u_n\|_{ \mu}}\right\rangle\right|\le M.
\end{equation*}
Thus, if~$\vartheta$ is as in~\eqref{AR3}, we have that
\begin{equation}\label{c11}
\vartheta I(u_n) - \langle I'(u_n), u_n\rangle \le M(\vartheta +\|u_n\|_{ \mu}).
\end{equation}

Moreover, we take~$r$ as in assumption~\eqref{AR3} and notice that there are two possible cases. On the one hand, if~$r=0$, then
\begin{equation}\label{cnaoj}
\int_\Omega \Big(f(x, u_n(x)) u_n(x) -\vartheta F(x, u_n(x))\Big)\,dx >0.
\end{equation}
On the other hand, if~$r>0$, then, by using Lemma~\ref{SV12} with~$\varepsilon:=1$, we find that
\begin{equation}\label{ler}
\begin{split}&
\left|\;\int_{\Omega\cap\{|u_n|\le r\}}  \Big(f(x, u_n(x)) u_n(x) 
-\vartheta F(x, u_n(x))\Big) \,dx\right|\\
&\qquad\qquad\le \Big(\widehat pr^{\widehat p} +q\delta(1) r^q 
+\vartheta r^{\widehat p} +\vartheta\delta(1) r^q\Big)|\Omega|=:C_r.
\end{split}
\end{equation}

Now, by~\eqref{functionaln}, \eqref{derfini7952qlwhfqor3u2t8y} and 
either~\eqref{cnaoj} if~$r=0$ or~\eqref{ler} if~$r>0$, we see that
\begin{equation*}
\begin{split}&
\vartheta I(u_n) - \langle I'(u_n), u_n\rangle \\
=\;&\sum_{k=1}^m\frac{1}{2}\left(\frac{\vartheta}{p_k} -1\right)[u_n]_{s_k,p_k}^{p_k}+\int_\Omega \Big(f(x, u_n(x))u_n(x)
-\vartheta F(x, u_n(x))\Big) \,dx
\\
\ge\;& \sum_{k=1}^m\frac{1}{2}\left(\frac{\vartheta}{p_k} -1\right)[u_n]_{s_k,p_k}^{p_k}+\int_{\Omega\cap\{|u_n|\le r\}}  \Big(f(x, u_n(x)) u_n(x) -\vartheta F(x, u_n(x))\Big) \,dx\\
\geq\;&\sum_{k=1}^m \frac{1}{2} \left(\frac{\vartheta}{p_k} -1\right)[u_n]_{s_k,p_k}^{p_k}
-C_r  \\
\geq\; &\frac{1}{2}\left(\frac{\vartheta}{\widehat{p}} -1\right) \sum_{k=1}^m[u_n]_{s_k,p_k}^{p_k}
-C_r.
\end{split}
\end{equation*}
Combining this with~\eqref{c11}, we obtain that
\begin{equation}\label{sperolimitata}
\sum_{k=1}^m[u_n]_{s_k,p_k}^{p_k}
\leq \frac{2\widehat{p}M(\vartheta +\|u_n\|_{ \mu})+2\widehat{p}C_r}{(\vartheta-\widehat{p})}.
\end{equation}

We claim that~\eqref{sperolimitata} implies that~$u_n$ is bounded 
in~$\mathcal{X}_\mu(\Omega)$.
In order to check this, assume by contradiction that 
\[
\lim_{n\to+\infty}\|u_n\|_\mu =+\infty.
\]
Then, defining the set
\[
K:=\Big\{k\in\{1,\cdots,m\}\; \mbox{s.t. there exists~$n_0$
s.t. }\; [u_n]_{s_k,p_k}\geq 1 \mbox{ for every } n\geq n_0
\Big\},
\]
we have that
\begin{equation}\label{sommaKdiverge}
\lim_{n\to+\infty}\sum_{k\in K} [u_n]_{s_k,p_k}=+\infty.
\end{equation}

Moreover, using~\eqref{elementaryinequality} and denoting by~$|K|$ the cardinality of the set~$K$ and 
\[
p_{min}:=\min_{k\in K}\{p_k\},
\]we see that
\begin{equation}\label{sommaKdiverge1}
\sum_{k=1}^m[u_n]_{s_k,p_k}^{p_k}
\ge \sum_{k\in K}[u_n]_{s_k,p_k}^{p_k}
\ge \sum_{k\in K}[u_n]_{s_k,p_k}^{p_{min}}
\ge \frac{1}{|K|^{p_{min}-1}}
\left(\sum_{k\in K}[u_n]_{s_k,p_k}\right)^{p_{min}}.
\end{equation}
Also,
\begin{equation}\label{sommaKdiverge2}
\|u_n\|_\mu=\sum_{k=1}^m[u_n]_{s_k,p_k}
\leq \sum_{k\in K}[u_n]_{s_k,p_k}+m-|K|.
\end{equation}

As a consequence, combining~\eqref{sommaKdiverge},
\eqref{sommaKdiverge1} and~\eqref{sommaKdiverge2},
\[
\lim_{n\to+\infty}\frac{1}{\|u_n\|_\mu}
\sum_{k=1}^m[u_n]_{s_k,p_k}^{p_k}
\geq \frac{1}{|K|^{p_{min}-1}} \lim_{n\to+\infty}
\dfrac{\displaystyle \left(\sum_{k\in K}[u_n]_{s_k,p_k}\right)^{p_{min}}}{\displaystyle \sum_{k\in K}[u_n]_{s_k,p_k}+m-|K|}=+\infty,
\]
which is in contradiction with~\eqref{sperolimitata}.
\end{proof}

We now prove that every bounded Palais-Smale sequence converges strongly in~$\mathcal{X}_\mu(\Omega)$.

\begin{proposition}\label{PropPS2}
Let~$f$ satisfy~\eqref{AR1}. Let~$u_n$ be a bounded sequence in~$\mathcal{X}_{ \mu}(\Omega)$ such that~\eqref{bdd2} holds true.

Then, there exists~$u\in\mathcal{X}_{ \mu}(\Omega)$ such that
\[
u_n\to u\quad\mbox{ in } \mathcal{X}_{ \mu}(\Omega).
\]
\end{proposition}

\begin{proof}
Since~$u_n$ is bounded in~$ \mathcal{X}_{ \mu}(\Omega)$, by Corollary~\ref{corembedding} there exists~$u\in\mathcal{X}_{ \mu}(\Omega)$ such that, up to subsequences, 
\begin{align}\label{weakconv}
&u_n\rightharpoonup u \quad\mbox{ in }  \mathcal{X}_{ \mu}(\Omega)\\ \label{convlp}
{\mbox{and }}\qquad& u_n\to u \quad\mbox{ in }  L^r(\Omega), \quad\mbox{ for any } r\in [1, (p_\sharp)^*_{s_\sharp}).
\end{align}

Now, by~\eqref{bdd2} and~\eqref{weakconv}, we have that
\[
\lim_{n\to +\infty}|\langle I'(u_n), u_n -u\rangle| = 0,
\]
namely
\begin{equation}\label{dwputgvha2345678ghergdk}
\begin{split}
&\lim_{n\to +\infty}\Bigg|
 \sum_{k=1}^m \frac{c_{N, s_k, p_k}}{2} \iint_{\R^{2N}} \frac{(u_n(x)-u_n(y))^{p_k -2} (u_n(x)-u_n(y))((u_n-u)(x)-(u_n-u)(y))}{|x-y|^{N+s_k p_k}}\, dx\, dy\\
&\qquad\qquad\qquad-\int_\Omega f(x, u_n(x)) (u_n-u)(x)\, dx\Bigg|=0.
\end{split}
\end{equation}

We recall that, by~\eqref{AR1}, 
\[
\left|\int_\Omega f(x, u_n(x)) (u_n-u)(x) \, dx\right|\le \int_\Omega (a_1 +a_2 |u_n(x)|^{q-1}) |u_n(x) - u(x)| \, dx.
\]
Moreover, since~$a_1 +a_2 |u_n|^{q-1} \in L^{\frac{q}{q-1}}(\Omega)$
(with uniformly bounded norm), by the H\"older inequality and~\eqref{convlp} we infer that
\begin{eqnarray*}&&
\lim_{n\to +\infty} \left|\int_\Omega f(x, u_n(x)) (u_n-u)(x) \, dx\right|
\le  \lim_{n\to +\infty} \int_\Omega (a_1 +a_2 |u_n(x)|^{q-1}) |u_n(x) - u(x)| \, dx\\&&\qquad\qquad
\le \lim_{n\to +\infty}\left(\int_\Omega (a_1 +a_2 |u_n(x)|^{q-1}) ^{\frac{q}{q-1}}\,dx\right)^{\frac{q-1}{q}}
\left(\int_\Omega  |u_n(x) - u(x)|^q\,dx\right)^{\frac1q}\\&&\qquad\qquad
\le C \lim_{n\to +\infty}\left(\int_\Omega  |u_n(x) - u(x)|^q\,dx\right)^{\frac1q}
=0.
\end{eqnarray*}

{F}rom this and~\eqref{dwputgvha2345678ghergdk}
we thus have that 
\begin{equation}\label{lim2}
\begin{split}
&\lim_{n\to +\infty}
 \sum_{k=1}^m \frac{c_{N, s_k, p_k}}{2} \iint_{\R^{2N}} \frac{(u_n(x)-u_n(y))^{p_k -2} (u_n(x)-u_n(y))((u_n-u)(x)-(u_n-u)(y))}{|x-y|^{N+s_k p_k}}\, dx\, dy
\\&\qquad\qquad\qquad=0.
\end{split}
\end{equation}

Since for any~$k$ the seminorm~$[\cdot]_{s_k,p_k}$ is convex,
we find that
\[
\begin{aligned}&
[u]_{s_k,p_k}^{p_k}-[u_n]_{s_k,p_k}^{p_k}\\
\geq\;& 
p_kc_{N, s_k, p_k} \iint_{\R^{2N}} \frac{(u_n(x)-u_n(y))^{p_k -2} (u_n(x)-u_n(y))((u-u_n)(x)-(u-u_n)(y))}{|x-y|^{N+s_k p_k}}\, dx\, dy.
\end{aligned}
\]
Summing over~$k$ and passing to the limit, it follows from~\eqref{lim2} that
\begin{equation*}
\sum_{k=1}^m[u]_{s_k,p_k}^{p_k}\geq 
\limsup_{n\to +\infty}\sum_{k=1}^m[u_n]_{s_k,p_k}^{p_k}.
\end{equation*}
Moreover, from Fatou's Lemma,
\[
\liminf_{n\to +\infty}\sum_{k=1}^m[u_n]_{s_k,p_k}^{p_k}
\geq \sum_{k=1}^m[u]_{s_k,p_k}^{p_k}.
\]
Accordingly,
\begin{equation}\label{lim4}
\lim_{n\to +\infty}\sum_{k=1}^m[u_n]_{s_k,p_k}^{p_k}
=\sum_{k=1}^m[u]_{s_k,p_k}^{p_k}.
\end{equation}

Since~$\|u_n\|_\mu$ is bounded, for any~$k\in\{1,\dots,m\}$ we have that~$[u_n]_{s_k,p_k}$ is bounded, and therefore it converges up to a
subsequence. As a result, we can write~\eqref{lim4} as
\begin{equation}\label{lim5}
\sum_{k=1}^m\left(\lim_{n\to +\infty}[u_n]_{s_k,p_k}^{p_k}-[u]_{s_k,p_k}^{p_k} \right)=0.
\end{equation}

Also, for any~$k\in\{1,\dots,m\}$, Fatou's Lemma implies that
\begin{equation*}
\lim_{n\to +\infty}[u_n]_{s_k,p_k}^{p_k}=
\liminf_{n\to +\infty}[u_n]_{s_k,p_k}^{p_k}\geq [u]_{s_k,p_k}^{p_k}.
\end{equation*}
Combining this with~\eqref{lim5}, we conclude that, for any~$k\in\{1,\dots,m\}$,
\[
\lim_{n\to +\infty}[u_n]_{s_k,p_k}^{p_k}=[u]_{s_k,p_k}^{p_k},
\]
that is
\[
\lim_{n\to +\infty}[u_n]_{s_k,p_k}=[u]_{s_k,p_k}.
\]
Thus, summing over~$k$,
\[
\lim_{n\to +\infty}\|u_n\|_\mu=
\lim_{n\to +\infty}\sum_{k=1}^m [u_n]_{s_k,p_k}=
\sum_{k=1}^m \lim_{n\to +\infty} [u_n]_{s_k,p_k}=
\sum_{k=1}^m [u]_{s_k,p_k}
=\|u\|_\mu.
\]

Now, from Proposition~\ref{uniformeconvesso} we know that~$\mathcal{X}_\mu(\Omega)$ is uniformly convex, and thus~$u_n$ converges strongly to~$u$ in~$\mathcal{X}_\mu(\Omega)$, 
as desired.
\end{proof}

\begin{proof}[Proof of the Theorem~\ref{MPT}]
By Propositions~\ref{geo1} and~\ref{geo2}, we infer that the functional~$I$ satisfies the geometric properties of the Mountain Pass Theorem.
Moreover, the validity of the Palais-Smale condition at any level~$c\in\R$ is provided by Propositions~\ref{PropPS1} and~\ref{PropPS2}. 
Thus, we can apply the Mountain Pass Theorem and conclude that there exists a critical point~$u\in\mathcal{X}_\mu(\Omega)$ of~$I$ such that
\[
I(u)\geq \beta >0=I(0).
\]
Hence, in particular, $u\neq 0$, as desired.
\end{proof}

\section{Some applications of Theorems~\ref{Weierstrass} and~\ref{MPT}}\label{sec4}

In this section, we provide some explicit examples of measures satisfying~\eqref{muvbdshail}, \eqref{anchequestat4390}, \eqref{ipotesimus} and~\eqref{ipotesimup} and we provide 
some applications of Theorems~\ref{Weierstrass} and~\ref{MPT}.


\begin{corollary}\label{corstrano1}
Let~$l$, $n$, $m\in\N\setminus\{0\}$.
For any~$i\in\{1,\dots,n\}$, let~$s_i\in[0,1]$ and~$\overline s:=\max\{s_1,\dots,s_n\}$.
For any~$k\in\{1,\dots,l\}$, let~$\tilde s_k\in[0,\overline s)$
and~$\alpha_k\in\R$.

For any~$j\in\{1,\dots, m\}$, let~$p_1,\dots,p_m\in(1,N)$ and~$\beta_j\in(0,+\infty)$.

Let~$g\in L^{\widehat{p}'}(\Omega)$. 

Then, there exists~$\alpha_0>0$ such that if
\begin{equation*}
\sum_{k=1}^{l}\alpha_{k}< \alpha_0,
\end{equation*}
then the problem
\[
\begin{cases}
\displaystyle\sum_{{i=1,\dots,n}\atop{j=1,\dots,m }} \beta_j(-\Delta)^{s_i}_{p_j} u -\sum_{{k=1,\dots,l}\atop{j=1,\dots,m }}\alpha_{k}\beta_j(-\Delta)^{s_k}_{p_j} u= g(x) &{\mbox{ in }}\Omega,\\
u=0&{\mbox{ in }}\R^N\setminus\Omega
\end{cases}
\]
admits a nontrivial solution corresponding to a global minimizer for the associated energy functional.

In addition, 
\[
\begin{cases}
\displaystyle\sum_{{i=1,\dots,n}\atop{j=1,\dots,m }} \beta_j
(-\Delta)^{s_i}_{p_j} u = g(x) &{\mbox{ in }}\Omega,\\
u=0&{\mbox{ in }}\R^N\setminus\Omega
\end{cases}
\]
has a unique solution.
\end{corollary}

\begin{proof}
We set
\[
\mu:= \sum_{{i=1,\dots,n}\atop{j=1,\dots,m }}\beta_j
\delta_{(s_i, p_j)} - \sum_{{k=1,\dots,l}\atop{j=1,\dots,m }}
\alpha_{k}\beta_j\,\delta_{(s_k, p_j)}.
\]
We point out that~$\mu$ is in the form~\eqref{muvbdshail} and that~\eqref{anchequestat4390} and~\eqref{ipotesimup} are satisfied.

Moreover, to check~\eqref{ipotesimus}, we observe that
$$\mu_s^-([0,\overline s)) =\sum_{k=1}^l \alpha_{k}=
\mu_s^+([\overline s,1])
\sum_{k=1}^l \alpha_{k}. $$
Thus, \eqref{ipotesimus} is satisfied with
$$\gamma:=\sum_{k=1}^l \alpha_{k}.$$
Hence, we derive the desired result from Theorem~\ref{Weierstrass}.
\end{proof}

\begin{proof}[Proof of Corollary~\ref{cor1}]
The results are a direct consequence of Corollary~\ref{corstrano1}, by taking~$n=m=l=1$, $\widehat p = p$, $\overline s=s_1$ and~$\beta=1$.
\end{proof}

\begin{corollary}\label{corstrano2}
Let~$n$, $m\in\N\setminus\{0\}$.
For any~$i\in\{1,\dots,n\}$, let~$s_i\in[0,1]$ and~$\overline s:=\max\{s_1,\dots,s_n\}$.
For any~$k\in\N$, let~$\tilde s_k\in[0,\overline s)$
and~$\alpha_k\in\R$.

For any~$j\in\{1,\dots, m\}$, let~$p_1,\dots,p_m\in(1,N)$ and~$\beta_j\in(0,+\infty)$.

Let~$g\in L^{\widehat{p}'}(\Omega)$. 

Then, there exists~$\alpha_0>0$ such that if
\[
\sum_{k=1}^{+\infty} \alpha_{k}< \alpha_0,
\]
then the problem
\[
\begin{cases}
\displaystyle\sum_{{i=1,\dots,n}\atop{j=1,\dots,m}} \beta_j(-\Delta)^{s_i}_{p_j} u - \sum_{{k\in\N}\atop{j=1,\dots,m }}\alpha_{k}\beta_{ j}(-\Delta)^{s_k}_{p_j} u= g(x) &{\mbox{ in }}\Omega,\\
u=0&{\mbox{ in }}\R^N\setminus\Omega
\end{cases}
\]
admits a nontrivial solution corresponding to a global minimizer for the associated energy functional. 
\end{corollary}

\begin{proof}
We set
\[
\mu:= \sum_{{i=1,\dots,n}\atop{j=1,\dots,m}} \beta_j\delta_{(s_i, p_j)} - \sum_{{k\in\N}\atop{j=1,\dots,m}}\alpha_{k}\beta_j\,\delta_{(s_k, p_j)}.
\]
We point out that~$\mu$ is in the form~\eqref{muvbdshail} and that~\eqref{anchequestat4390} and~\eqref{ipotesimup} are satisfied.

Moreover, to check~\eqref{ipotesimus}, we observe that
$$\mu_s^-([0,\overline s)) =\sum_{k=1}^{+\infty} \alpha_{k}=
\mu_s^+([\overline s,1])
\sum_{k=1}^{+\infty} \alpha_{k} . $$
Thus, \eqref{ipotesimus} is satisfied with
$$\gamma:=\sum_{k=1}^{+\infty} \alpha_{k}.$$
The desired result then follows directly from Theorem~\ref{Weierstrass}.
\end{proof}

\begin{proof}[Proof of Corollary~\ref{cor3}]
The result is a direct consequence of Corollary~\ref{corstrano2}, by taking~$n=m=1$, $\widehat p = p$, $\overline s=1$ and~$\beta=1$.
\end{proof}

\begin{corollary}\label{corstrano3}
Let~$n$, $m\in\N\setminus\{0\}$.
For any~$i\in\{1,\dots,n\}$, let~$s_i\in[0,1]$ and~$\overline s:=\max\{s_1,\dots,s_n\}$.
For any~$j\in\{1,\dots, m\}$, let~$p_1,\dots,p_m\in(1,N)$ and~$\beta_j\in(0,+\infty)$.

Let~$\omega:[0,1]\to \R$ be any measurable function.
Let~$g\in L^{\widehat{p}'}(\Omega)$. 

Then, there exists~$\omega_0>0$ such that if
\[
\int_0^{\overline s} \omega(s)\,ds< \omega_0,
\]
then problem
\[
\begin{cases}
\displaystyle\sum_{{i=1,\dots,n}\atop{j=1,\dots,m}} \beta_j
(-\Delta)^{s_i}_{p_j} u -  \sum_{j=1}^m 
\beta_{j}\int_0^{\overline s} (-\Delta)^{s}_{p_j} u \, \omega(s)\, ds= g(x) &{\mbox{ in }}\Omega,\\
u=0&{\mbox{ in }}\R^N\setminus\Omega
\end{cases}
\]
admits a nontrivial solution corresponding to a global minimizer for the associated energy functional.
\end{corollary}

\begin{proof}
We write~$\mu=\mu^+-\mu^-$, where
\[
\mu^+:= \sum_{{i=1,\dots,n}\atop{j=1,\dots,m}} \beta_j \delta_{(s_i, p_j)}\qquad\mbox{and}\qquad \mu^-:= \nu_s \, \sum_{j=1}^m \beta_{j}\,\delta_{p_j}, 
\]
being~$\nu_s$ such that~$
d\nu_s = \omega(s)\, ds$.

We observe that~$\mu$ is in the form~\eqref{muvbdshail} and that~\eqref{anchequestat4390} and~\eqref{ipotesimup} are satisfied.

Moreover, we see that
$$\mu_s^-([0,\overline s)) =\int_0^{\overline s}\omega(s)\,ds=
\mu_s^+([\overline s,1])
\int_0^{\overline s}\omega(s)\,ds, $$
and therefore~\eqref{ipotesimus} is satisfied with
$$\gamma:=\int_0^{\overline s}\omega(s)\,ds.$$
The desired result then follows from Theorem~\ref{Weierstrass}.
\end{proof}

\begin{proof}[Proof of Corollary~\ref{cor4}]
The result plainly follows by Corollary~\ref{corstrano3}, by taking~$n=m=1$,
$\widehat p = p$, $\overline s=1$ and~$\beta=1$.
\end{proof}

Finally, we prove Corollaries~\ref{cor5} and~\ref{cor6},
by exploiting the general statement in Theorem~\ref{MPT}.

\begin{proof}[Proof of Corollary~\ref{cor5}]
We set~$m=2$, so that~$\mu=\mu^+$ reduces to
\[
\mu:=\delta_{(s_1, p_1)} + \delta_{(s_2, p_2)}.
\]
In this case, the critical exponent is given by
\[
(p_\sharp)^*_{s_\sharp}:= \max\left\{\frac{N p_1}{N- s_1 p_1}, \frac{N p_2}{N- s_2 p_2}\right\}.
\]
Hence, the desired result follows from Theorem~\ref{MPT}.
\end{proof}

\begin{proof}[Proof of Corollary~\ref{cor6}]
In this case, $\mu=\mu^+$ is given by
\[
\mu:= \sum_{k=1}^m\delta_{(s_k, p_k)}.
\]
Moreover, the critical exponent is the one provided in~\eqref{criticocostante}. Thus, the desired result is a consequence of Theorem~\ref{MPT}.
\end{proof}

\begin{appendix}
\section{Discussion on the well-posedness of the energy functional}\label{appendice}

In order to deal with problem~\eqref{P1}, we restrict ourselves to
the case of a measure of the form~\eqref{muvbdshail}.
The reason for this is that
one can find measures for which the 
the functional
associated with problem~\eqref{P1} is not well defined, even though
the associated norm is finite on the functional space. 


More precisely, our aim is to 
find measures~$\mu$ and functions~$u\in \mathcal{X}_\mu$ such that
\[
\|u\|_\mu<+\infty \qquad \mbox{ and }\qquad 
\iint_\Sigma\frac{1}{2p} [u]_{s,p}^p\,d\mu^+(s,p)=+\infty,
\]
or, in other words, that the associated functional is not well 
defined on the space~$\mathcal{X}_\mu$.
On this matter, we provide two examples.
\medskip

First, we deal with the case
\[
\mu^+:= \sum_{k=0}^{+\infty} c_k\delta_{(s_k, p_k)}.
\]
Let~$\Omega=B_1\subset \R^N$, with~$N\geq 2$, and let
\[
s_k:=1,\, p_k:=N-\frac{1}{k}\, \mbox{ and } 
c_k:=k^{-1-\frac{1+\varepsilon}{p_k}}
\mbox{ for every }k\in \N \mbox{ and } \varepsilon\in (0,1).
\]
Then,
\[
\|u\|_\mu=\sum_{k=0}^{+\infty} c_k\|\nabla u\|_{L^{p_k}(B_1)}.
\]
We want to show that there exists some~$u$ with~$\|u\|_\mu<+\infty$
and
\[
\iint_\Sigma\frac{1}{2p} [u]_{s,p}^p\,d\mu^+(s,p)=
\sum_{k=0}^{+\infty} \frac{c_k}{2p_k}\|\nabla u\|_{L^{p_k}(B_1)}^{p_k}=+\infty.
\]
To do this, consider the function~$u(x):=\log|x|$, so that~$\nabla u(x)=x/|x|^2$. Clearly,
\begin{eqnarray*}&&
\|\nabla u\|_{L^{p_k}(B_1)}=
\left(\;\int_{B_1}\frac{dx}{|x|^{p_k}} \right)^\frac{1}{p_k}
=\omega_{N-1}^\frac{1}{p_k}
\left(\int_0^1\rho^{N-p_k-1}\,d\rho  \right)^\frac{1}{p_k}
\\&&\qquad\qquad=\frac{\omega_{N-1}^\frac{1}{p_k}}{(N-p_k)^{\frac{1}{p_k}}}
=\omega_{N-1}^\frac{1}{N-1/k}k^\frac{1}{N-1/k}
\end{eqnarray*}
Thus,
\[
\|u\|_\mu=
\sum_{k=0}^{+\infty} c_k \|\nabla u\|_{L^{p_k}(B_1)}
=\sum_{k=0}^{+\infty} k^{-1-\frac{1+\varepsilon}{N-1/k}}
\omega_{N-1}^\frac{1}{N-1/k}k^\frac{1}{N-1/k}
=\sum_{k=0}^{+\infty} 
\omega_{N-1}^\frac{1}{N-1/k}k^{-1-\frac{\varepsilon}{N-1/k}}<+\infty.
\]
On the other hand,
\[
\sum_{k=0}^{+\infty} \frac{c_k}{2p_k}\|\nabla u\|_{L^{p_k}(B_1)}^{p_k}
\geq \frac{1}{2N}\sum_{k=0}^{+\infty} k^{-1-\frac{1+\varepsilon}{N-1/k}}
\omega_{N-1}k
=\frac{\omega_{N-1}}{2N}\sum_{k=0}^{+\infty}
k^{-\frac{1+\varepsilon}{N-1/k}}=+\infty,
\]
as desired.
\medskip

For the second example, we consider
\[
\mu^+= \delta_1 \times \nu_p.
\]
Here, $\nu_p$ is a measure defined on~$(1,N)$ such that~$
d\nu_p=h(p)\,dp$
for some measurable function~$h:(1,N) \to \R$.

Let~$\Omega=B_1\subset \R^N$, with~$N\geq 2$, and let~$h(p):=2p$.
Then,
\[
\|u\|_\mu=\int_1^N \|\nabla u\|_{L^{p}(B_1)}\,dp.
\]
In this case, we want to show that there exists some~$u$ with~$\|u\|_\mu<+\infty $ and 
\[
\iint_\Sigma\frac{1}{2p} [u]_{s,p}^p\,d\mu^+(s,p)=
\int_1^N \|\nabla u\|_{L^{p}(B_1)}^p\,dp=+\infty.
\]
In order to do this, we argue as in the first example
and consider the function~$u(x):=\log |x|$.
Now,
\[
\|\nabla u\|_{L^{p}(B_1)}=
\frac{\omega_{N-1}^\frac{1}{p}}{(N-p)^\frac{1}{p}}.
\]
Thus,
\[
\|u\|_\mu=\int_1^N \|\nabla u\|_{L^{p}(B_1)}\,dp
=\int_1^N\frac{\omega_{N-1}^\frac{1}{p}}{(N-p)^{\frac{1}{p}}}\,dp.
\]
Clearly,
\begin{equation}\label{controesempioappendice1}
\int_1^{N-\frac12}\frac{\omega_{N-1}^\frac{1}{p}}{(N-p)^{\frac{1}{p}}}\,dp<+\infty.
\end{equation}
Moreover, if~$p\in (N-\frac{1}{2},N)$ we have that
\[
(N-p)^{\frac{1}{N-\frac{1}{2}}}\leq (N-p)^{\frac{1}{p}}\leq (N-p)^{\frac{1}{N}},
\]
and so 
\begin{equation*}
\int_{N-\frac12}^N\frac{\omega_{N-1}^\frac{1}{p}}{(N-p)^{\frac{1}{p}}}\,dp
\leq \int_{N-\frac12}^N\frac{\omega_{N-1}^\frac{1}{p}}{(N-p)^{\frac{1}{N-\frac{1}{2}}}}\,dp<+\infty.
\end{equation*}
{F}rom this and~\eqref{controesempioappendice1} we infer that~$\|u\|_\mu<+\infty$.

In addition,
\[
\int_1^N \|\nabla u\|_{L^{p}(B_1)}^p\,dp
=\int_1^N \frac{\omega_{N-1}}{(N-p)}\,dp=+\infty,
\]
as desired.

\section*{Acknowledgements} 
All the authors are members of the Australian Mathematical Society (AustMS). CS and EPL are members of the INdAM--GNAMPA.

This work has been supported by the Australian Laureate Fellowship FL190100081.
\end{appendix}

\vfill

\end{document}